\documentclass{article}

\usepackage{arxiv}
\usepackage{changepage}
\usepackage{hyperref}      
\usepackage{amsfonts}       
\usepackage{microtype}      
\usepackage{graphicx}
\usepackage{tikz}
\usepackage{mathtools}
\usetikzlibrary{matrix,positioning,calc}
\usepackage[caption=false]{subfig} 
\usepackage{caption}
\usepackage{amsthm, amssymb}
\theoremstyle{definition}
\newtheorem{definition}{Definition}[section]
\theoremstyle{plain}
\newtheorem{proposition}{Proposition}[section]
\newtheorem{lemma}{Lemma}[section]
\newtheorem{theorem}{Theorem}[section]
\newtheorem{corollary}{Corollary}[section]
\theoremstyle{remark}

\let\oldthebibliography\thebibliography
\let\endoldthebibliography\endthebibliography
\renewenvironment{thebibliography}[1]{%
  \oldthebibliography{#1}%
  \small 
  \setlength{\itemsep}{3pt}%
  \setlength{\parskip}{3pt}%
}{%
  \endoldthebibliography%
}

\title{Persistent homology of featured time series data and its applications\thanks{Preprint. Under review.}}

\author{
 \large Eunwoo Heo\thanks{Department of Mathematics, Pohang University of Science and Technology (POSTECH), Pohang, South Korea (hew0920@postech.ac.kr)} \And \large 
 Jae-Hun Jung\thanks{Department of Mathematics, POSTECH, Pohang, South Korea (jung153@postech.ac.kr)}
 }

\begin{document}

\maketitle
\begin{abstract}
Recent studies have actively employed persistent homology (PH), a topological data analysis technique, to analyze the topological information in time series data. 
Many successful studies have utilized graph representations of time series data for PH calculation.
Given the diverse nature of time series data, it is crucial to have mechanisms that can adjust the PH calculations by incorporating domain-specific knowledge. 
In this context, we introduce a methodology that allows the adjustment of PH calculations by reflecting relevant domain knowledge in specific fields.
We introduce the concept of \textit{featured time series}, which is the pair of a time series augmented with specific features such as domain knowledge, and an \textit{influence vector} that assigns a value to each feature to fine-tune the results of the PH. 
We then prove the stability theorem of the proposed method, which states that adjusting the influence vectors grants stability to the PH calculations.
The proposed approach enables the tailored analysis of a time series based on the graph representation methodology, which makes it applicable to real-world domains. 
We consider two examples to verify the proposed method's advantages: anomaly detection of stock data and topological analysis of music data.
\end{abstract}


\section{Introduction}
Data that can be represented as one-dimensional variables are called time series data; this data format is simple and widely prevalent across various fields.
Time series data analysis has been a long-standing field, including anomaly detection~\cite{braei2020anomaly, blazquez2021review}, forecasting~\cite{lim2021time, sezer2020financial, torres2021deep, lara2021experimental}, classification~\cite{bagnall2017great, susto2018time, ismail2019deep, abanda2019review}, and clustering~\cite{aghabozorgi2015time, maharaj2019time, ali2019clustering}.
Despite being extensively researched, time series analysis remains a challenging problem.
Recent advancements in topological data analysis (TDA) have led to a surge in research on analyzing time series data using topological information.
A prominent method within TDA is PH, wherein the duration for which a \(p\)-dimensional homology persists is calculated as the given data undergoes sequential transformations, thereby revealing the underlying topological structure.
Research analyzing time series data using PH can be broadly classified into two categories: studies utilizing delay embedding and those utilizing graph representation.

The delay embedding \(\phi_{\tau}\) for a time series \(T\), as described by Packard et al.~\cite{packard1980geometry}, is a function into \(\mathbb{R}^m\) defined by
\(
\phi_{\tau}(t) = (T(t), T(t-\tau), \ldots, T(t-(d-1)\tau)),
\)
where \(\tau\) is the delay lag and \(m\) is the embedding dimension. This embedding method is employed to transform time series data into point clouds (PCs) in Euclidean space \(\mathbb{R}^m\) for PH analysis. 
Takens' embedding theorem~\cite{takens2006detecting} ensures that this transformation preserves the topological information of the data sampled from hidden dynamical systems.
The initial approach was used to analyze dynamical systems in~\cite{Skraba2012TopologicalAO}, where the time series sampled from these systems were transformed into PCs for PH analysis. 
Subsequent research~\cite{perea2015sliding, perea2016persistent} expanded this foundation by affirming the effectiveness of PH in identifying the periodicity within a time series and extending its application to quasi-periodic systems.
Research on the early detection of critical transitions~\cite{berwald2013critical} based on~\cite{scheffer2009early} and analysis of the stability of dynamical systems~\cite{khasawneh2014stability} was conducted using PH and delay embedding. 
Moreover, PH and delay embedding methodologies were adopted for clustering~\cite{pereira2015persistent} and classification~\cite{tran2019topological} of time series data.
The effectiveness of PH in multivariate time series analysis for machine learning tasks such as room occupancy detection was demonstrated in~\cite{wu2022topological}.
The versatility of PH has been further highlighted by demonstrating its broad applicability to diverse fields, e.g., identifying wheezes in signals~\cite{emrani2014persistent}, detecting financial crises~\cite{gidea2018topological,ismail2022early}, classifying motions using motion-capture data~\cite{venkataraman2016persistent}, and distinguishing periodic biological signals from nonperiodic synthetic signals~\cite{perea2015sw1pers}. 

On the other hand, in addition to embedding methods,
several studies have successfully used graph representation of time series data with PH.
The graph representation \( G = (V, E) \) of a time series \( T \) consists of a vertex set \( V \) and an edge set \( E \), which are derived from the data points and their relationships within \( T \).
This graph is then analyzed using PH to study the topological features of \( T \).
Notably, functional magnetic resonance imaging (fMRI) data of the brain were converted into graph data, called functional networks in~\cite{stolz2017persistent}, which were then analyzed using PH. 
In~\cite{gidea2017topological}, the early signs of critical transitions in financial time series were detected by converting the series into graph data, called correlation networks, for PH analysis. 
In text mining, PH has been applied to graphs representing the main characters in the text time series~\cite{gholizadeh2018topological}. 
A study proposed the construction of graphs, called music networks, for classifying Turkish makam music through PH analysis~\cite{aktas2019classification}. 
These music networks were used in~\cite{tran2023topological} to analyze traditional Korean music using PH and in~\cite{tran2024machine}
to study composition by integrating machine learning.
In addition, the time series data have been transformed into Tonnetz, proposed by Euler, for music classification through PH analysis~\cite{bergomi2020homological}. 
The authors of~\cite{mijangos2024musical} addressed the classification of music style by creating data called intervallic transition graphs from music time series and applying PH.
Graph representation of time series is also crucial in recent artificial intelligence-based time series analysis utilizing graph neural networks (GNNs). 
The graph representation of time series data was employed for forecasting~\cite{cao2020spectral}, anomaly detection~\cite{deng2021graph}, and classification~\cite{zha2022towards} with GNNs.

Given the extensive and diverse nature of time series data, in addition to delay embedding, various other embedding methods such as derivative embedding~\cite{packard1980geometry}, integral-differential embedding~\cite{gilmore1998topological}, and global principal value embedding~\cite{broomhead1986extracting} have been used. 
Recently, selecting appropriate embeddings from these various embeddings using PH was investigated in~\cite{tan2023selecting}. 
Similarly, choosing the appropriate graph representation for each research field is also crucial.
However, further research is still needed on how to adjust and select suitable graph representations.
In this context, our study makes the following contributions:
\begin{enumerate}
    \item We introduce a novel concept of \textit{featured time series data} created by adding domain knowledge (features) to the time series data and \textit{influence vectors} that assign a value to each feature. This allows the adjustment of graph representations to ensure suitability to the specific domain.
    \item We prove that adjusting the graph representations via the influence vectors provides stability to the PH calculation, thereby demonstrating the robustness of the proposed method.
\end{enumerate}

The remainder of this paper is organized as follows.
Section~\ref{sec:Frequency-based PH analysis of time series data} describes the frequency-based graph representation and methods for calculating the PH of the graphs.
We provide an example of potential information loss that can occur during the PH computation process.
In Section~\ref{sec:Introduction of Featured Time Series Data}, we introduce the novel concepts of featured time series data and influence vectors to analyze a time series by incorporating domain knowledge.
Furthermore, we show that the examples of information loss presented in Section~\ref{sec:Frequency-based PH analysis of time series data} can be addressed by adjusting the influence vectors. 
In Section~\ref{sec:Stability theorem for influence vectors}, we prove the theorem that asserts the stability property of the influence vectors when subjected to such adjustments. 
This theorem underscores the reliability and robustness of the influence vector adjustments in mitigating information loss during the graph representation process. 
In Section~\ref{sec:Applications and experiments}, we explore the effect of variations in the influence vectors on the analysis of real-world time series data based on previous research, including anomaly detection of stock data and topological analysis of music data.
In Section~\ref{sec:Conclusion}, a brief concluding remark is provided.
\section{Frequency-based PH analysis of time series data}\label{sec:Frequency-based PH analysis of time series data} 
In this study, we addressed the graph representation of time series data using the frequency of occurrence of observations in the time series data.

\subsection{Frequency-based graph construction}\label{subsec:Frequency-based graph construction}

Consider a time series \( T : \mathbb{T} \rightarrow \mathcal{X} \), where \( \mathbb{T} \) is assumed finite. 
Define the vertex or node set \( V \) as 
\[ V = \{\{T(t)\} \mid t \in \mathbb{T} \}. \]
Enumerating \( \mathbb{T} \) as \( \{ t_1, t_2, \ldots, t_n\} \), the edge set \( E \) is defined as 
\[ E = \{ \{T(t_i), T(t_{i+1} ) \} \mid 1 \leq i \leq n-1 ,  T(t_i) \neq T(t_{i+1} ) \}. \]
Let \( f_e \) represent \textit{the frequency of an edge} \( e \in E\) in $T$ given by
\[ f_e = \Big\lvert \{t_i \in \mathbb{T} \mid e = \{T(t_i), T(t_{i+1} ) \} \} \Big\rvert. \]
Here, \( f_e \) measures the total number of co-occurrences of the two nodes appearing side by side in $T$ associated with the edge $e$. 
Note that when defining the frequency $f_e$ as above, we did not consider these two nodes' specific order of appearance. 
The frequency of the edge $e$ increases whenever the two nodes are positioned adjacent to each other in $T$.
Define a weight function \( W_E \) on $E$ as $W_E(e) = f_e $ for any edge \( e \in E\).
This weight measures how strongly the associated nodes are connected to each other in $e \in E$.

\subsection{Distance on weighted graphs}\label{subsec:Distance on weighted graphs}
To compute the PH over the graph, we define the distance between two nodes in the graph. 
In this study, we considered the definition of natural distance, which is given below. 
However, note that the distance definition is not unique but might be the best definition depending on the problem.

\begin{definition}[Distance]\label{def:distance}
Let $G=(V, E, W_E)$ be a connected weighted graph. Define the distance between $v$ and $w$ in $V$ such that 
$d(v,w) = 0$ if $v=w$; otherwise, 
\[
    d(v,w)= \min_p\left\{ \sum_{e \in p} (W_E(e))^{-1} \Bigg\lvert
\text{$p$ is a path in $G$ connecting $v$ and $w$} \right\}. 
\]
\end{definition}

Note that the distance $d$ satisfies the metric conditions.
The proof is provided in Proposition~\ref{prop: d_ is metric}, where a more general case is presented.
For $d$, we utilize the reciprocal of the edge weight function \(W_E\). 
This implies that we consider the vertices \( v \) and \( w \) connected by the edge \( e \) as closer when the frequency \( f_{e} \) of the edge is higher.
Define a \textit{distance matrix} $A$ as $A=(a_{ij})$, where $a_{ij}=d(v_i, v_j)$ for every pair of vertices $v_i, v_j \in V$.

\subsection{Persistent homology of metric spaces}\label{subsec:Persistent homology of metric spaces}
We now consider computing PH given a metric space $(V,d)$. 
Suppose we have a metric space $(V,d)$ from graph $G=(V, E, W_E)$.
The power set $\mathcal{P}(V)$ of the vertex set $V$ is an abstract simplicial complex, denoted as $\mathbb{X}$.
We define the Vietoris-Rips (Rips) filtration function 
$h : \mathbb{X} \rightarrow \mathbb{R} $ as 
\[
h(\sigma) = \max \{ d(v_i,v_j) \mid \tau \subseteq \sigma , \text{1-simplex }\tau=\{v_i, v_j\}  \}
\]
whenever $p \geq 1$ for $p$-simplex $\sigma \in \mathbb{X}$. 
For any $0$-simplex $v$, we define $h(v) = 0$.
Denote $h^{-1}((-\infty,\epsilon])$ as $\mathbb{X}_{\epsilon}$ for $\epsilon \in \mathbb{R}$.
Then, for $\epsilon_n \ge \max_\sigma(h(\sigma))$ and $0=\epsilon_0 \le \epsilon_1 \le \epsilon_2 \le \cdots \le \epsilon_n$, 
\[
    \mathbb{X}_{\epsilon_0} \subseteq \mathbb{X}_{\epsilon_1} \subseteq \cdots \subseteq \mathbb{X}_{\epsilon_n} = \mathbb{X}
\]
forms the Rips filtration for an abstract simplicial complex $\mathbb{X}$ (see~\cite{vietoris1927hoheren,gromov1987hyperbolic}).
A simplicial complex can be associated with a sequence of abelian groups, termed chain groups. 
For each dimension \( p \), the \(p\)-chains \(C_p(X)\) are the formal sums of the \(p\)-simplices.
For each dimension \( p \), there exists a boundary map
\( \partial_p : C_p(\mathbb{X}) \to C_{p-1}(\mathbb{X}) \)
that maps each \( p \)-simplex to its boundary. 
A crucial property is that the boundary of the boundary of the simplices is always null, that is, $\partial_{p-1} \circ \partial_p = 0.$
Given the chain groups and boundary maps, the \textit{\( p \)th homology group} is defined as the quotient group 
\[
    H_p(\mathbb{X}) = \frac{\ker(\partial_p)}{\text{Im}(\partial_{p+1})}.
\]
This quotient captures the \( p \)-dimensional holes in a given topological space.
As the filtration progresses, these homology groups also change.
The PH serves as an effective tool for tracing such changes.
For each dimension \( p \) and filtration value \( \epsilon_i \), we compute 
$H_p(\mathbb{X}_{\epsilon_i}).$
As \( \epsilon_i \) changes, certain homological features (e.g., loops) may emerge or disappear.
A persistence diagram \cite{edelsbrunner2002topological} (PD) is a multi-set that visualizes the coordinates of these features on a plane.
The \( x \)-coordinate of a point in the PD indicates the birth of such a homological feature, and the \( y \)-coordinate indicates its death. 
The longer the persistence of the considered feature, the greater the deviation of the corresponding point from the diagonal on the plane of the PD.
Note that the computation of PH depends on the filtration choice, but Rips filtration is one of the most popular filtration methods. 
Open-source libraries used for computing the PH based on Rips filtration include JavaPlex~\cite{adams2014javaplex} written in Java but easily usable in MATLAB, GUDHI~\cite{maria2014gudhi} written in C++ but also accessible in Python, and the recently developed Ripser~\cite{bauer2021ripser}, which significantly reduces the computation time and memory usage when compared with the other methods.
Rip filtration allows the computation of PH not only on point clouds but also on general abstract simplicial complexes, especially graphs.
In \cite{attali2011vietoris}, the authors demonstrated that Rips filtration could approximate the topological shapes of the data's manifolds, as discussed in \cite{dey2022computational}.
 
\subsection{Discussion of information loss through an example}\label{subsec:Discussion of Information loss through an example}

Let us examine a simple example to see how PH computation for time series analysis of a specific field can result in information loss. 
Here, we discuss a simple case of anomaly detection in a time series.
Consider a time series \(T\) consisting of \(27\) timestamps for temperature changes, as shown in Figure~\ref{fig:Time_temperature_nofeat}. 
We can construct a graph from \(T\), as shown in Figure~\ref{fig:graph_distancematrix}. 
In the graph, each vertex represents a temperature value, and the value associated with the edge represents the co-occurring frequency of the two adjacent temperature values in \(T\). 
From the definition of the distance \(d\), the distance matrix is obtained as shown in Figure~\ref{fig:graph_distancematrix}. 
The matrix's rows and columns correspond to the temperature values associated with the vertices \(21, 22, 23, 24\) and \(25\). 
For example, the element at the position of $(2,3)$ in the matrix represents the distance between the vertices $22$ and $23$.
The complexity of computing the distance matrix depends on the complexity of the pathfinding process $O(\lvert E \rvert \log \lvert V \rvert)$, which can be efficiently calculated using the Dijkstra algorithm~\cite{dijkstra1959note}.

\begin{figure}[ht]
\centering
\includegraphics[width=\linewidth]{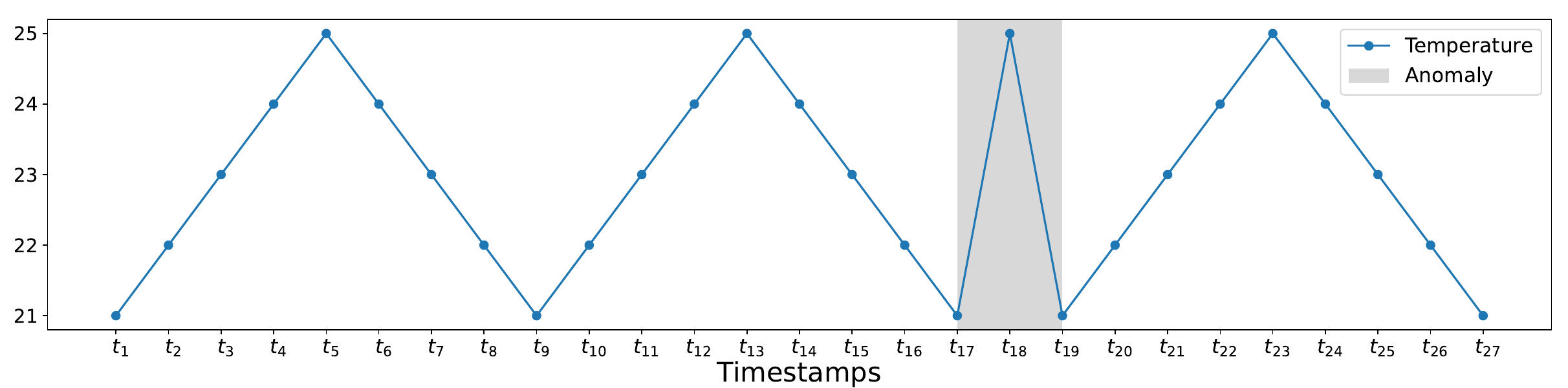}
\caption{A time series of temperature with the anomaly region shaded.}
\label{fig:Time_temperature_nofeat}
\end{figure}

\begin{figure}[ht]
\centering
\includegraphics[width=0.7\linewidth]{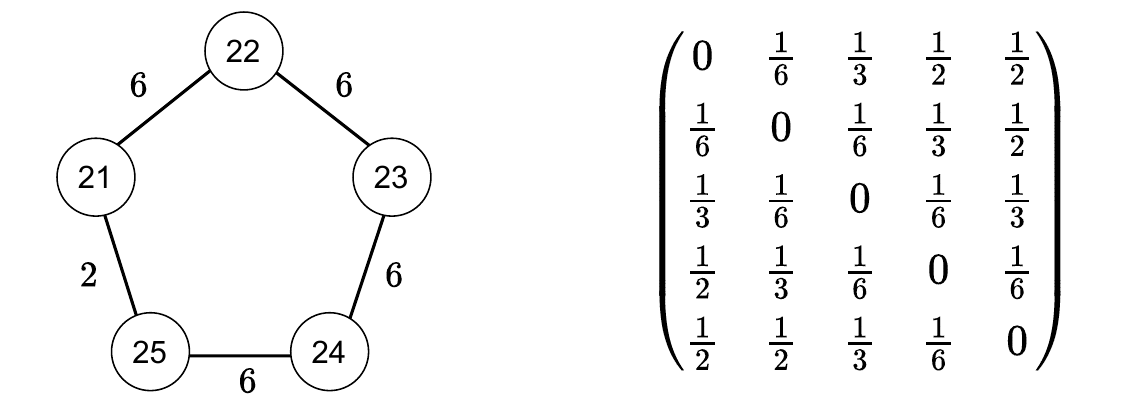}
\caption{Weighted graph $G=(V,E,W_E)$ of the time series $T$ in Figure~\ref{fig:Time_temperature_nofeat} (left) and the corresponding distance matrix (right).}
\label{fig:graph_distancematrix}
\end{figure}

\begin{figure}[ht]
    \centering
    \includegraphics[width=\linewidth]{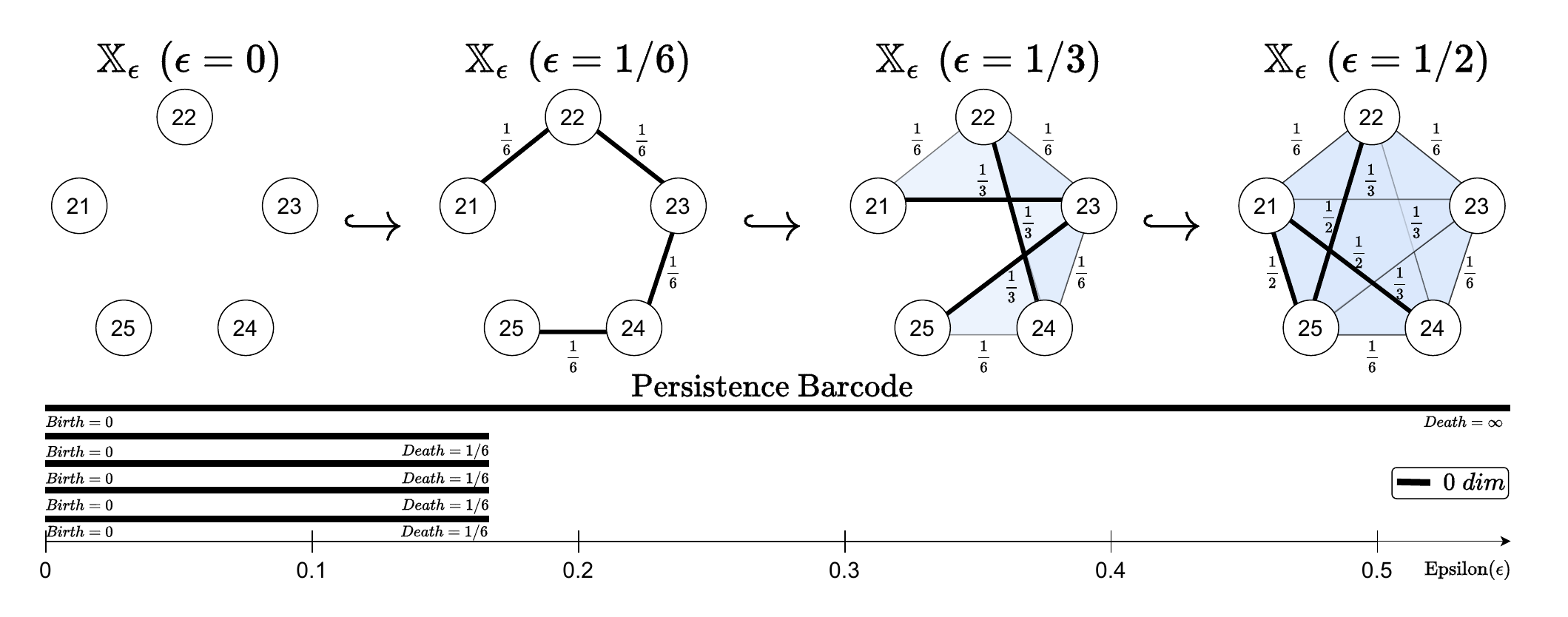}
    \caption{Rips filtration (Top) and its persistence barcode (Bottom).}
    \label{fig:Rips_complex}
\end{figure}

We construct the Rips filtration for the distance matrix and compute the persistence barcode.
The upper row of Figure~\ref{fig:Rips_complex} illustrates how the Rips complex changes in the Rips filtration as the filtration value $\epsilon$ varies.
The bottom row shows the corresponding \(0\)-dimensional persistence barcode.
The corresponding persistence barcode shown in Figure~\ref{fig:Rips_complex} illustrates the occurrence of information loss during filtration.
The time series \(T\) in Figure~\ref{fig:Time_temperature_nofeat} undergoes abnormal temperature changes in the shaded area, where temperatures (nodes) $21$ and $25$ are connected twice. 
As shown in Figure~\ref{fig:Rips_complex},
the information regarding the frequency between nodes \(21\) and \(25\) in the anomaly region and its distance \(\frac{1}{2}\) are not reflected in the computation of homology. 
The birth and death of the \(0\)-dimensional homology occur at \(0\), \(\frac{1}{6}\), and \(\infty\), with no occurrence at \(\frac{1}{2}\).  Moreover, the \(1\)-dimensional homology is not detected either.
This is because the interior of the pentagon is filled in advance at \(\frac{1}{3}\) preceding \(\frac{1}{2}\). 
Therefore, if the change in temperature from $21$ to $25$ is considered an anomaly, then it must be reflected in this loss of information. 
For example, instead of simply increasing the frequency by \(1\) at each appearance of \(21\) and \(25\) in $T$, we can assign a larger value.
This can generate $1$-dimensional homology because if the distance $d(21,25)$ becomes smaller than $\frac{1}{3}$, a \(1\)-cycle is formed before any interior edge of the pentagon is created.
In addition, if the distance $d(21,25)$ becomes less than \(\frac{1}{6}\), the edge \(\{21,25\}\) will be the first to emerge in the filtration; thus, $d(21,25)$ becomes the death time in the \(0\)-dimensional barcode.
This is discussed in detail in Example~\ref{example:improve_infoloss} later.
These examples underscore the critical need for domain-specific adjustments in analysis to accurately perform PH computations for time series data within particular fields. 
\section{Introduction of featured time series data}\label{sec:Introduction of Featured Time Series Data}

We propose the featured time series data below to enable adjustments in the PH calculation that reflect each field's domain knowledge.

\subsection{Featured time series data}\label{subsec:Introduction of Featured Time Series Data}

\begin{definition}[Feature set]
    Let $F^p$ denote finite sets, where $p=0,1$. 
    The Cartesian product of the power set of $F^p$ is called the \textit{ feature set}, denoted by $\mathcal{F}$.
    That is, $\mathcal{F}= \mathcal{P}(F^0) \times \mathcal{P}(F^1).$
    We call $F^p$ as the \textit{$p$th feature set}.
    The elements of $F^p$ are called the $p$th \textit{features}.
\end{definition}

As demonstrated below, the $p$th feature is associated with the \(p\)-simplex in a graph in our proposed method.

\begin{definition}[Influence vectors]
   Consider a non-negative real value function $g: F^0 \cup F^1 \cup \{ \emptyset^0, \emptyset^1 \} \rightarrow \mathbb{R}$, where the symbols $\emptyset^0$ and $\emptyset^1$ represent the states of the given time series without any feature. This function $g$ is called the influence vector with the size of $\lvert F^0 \rvert + \lvert F^1 \rvert + 2$.
   For any element $z$ in $ F^0 \cup F^1 \cup \{ \emptyset^0, \emptyset^1 \}$, the value $g(z)$ is referred to as the influence of $z$.
\end{definition}

\begin{definition}[Featured time series]\label{def:Featured time series}
    Let \(T: \mathbb{T} \rightarrow \mathcal{X}\) be a time series, \(\mathcal{F}\) be a feature set, and \(g\) be an influence vector.
    Consider a function \(\widehat{T}: \mathbb{T} \rightarrow \mathcal{X} \times \mathcal{F}\) such that $\widehat{T}(t) = (T(t), T_f(t)) \in \mathcal{X} \times \mathcal{F}$ for $t \in \mathbb{T}$, where $T_f:\mathbb{T} \rightarrow \mathcal{X}$ represents any function.
    The pair \((\widehat{T},g)\) is called a featured time series of $T$ by $g$. 
    The function $T_f$ is termed the feature component of the $\widehat{T}$.
\end{definition}

For example, let $F^0 = \{r_1, r_2, r_3\}$ and $F^1 = \{s_1, s_2\}$ represent the zeroth and first feature sets, respectively.
Consider a time series $T: \mathbb{T} \rightarrow \mathcal{X}$ such that $T(t_i) = x_{t_i}$
, where $\mathbb{T} = \{ t_i \mid i \in \mathbb{N} \}$.
The following sequence illustrates an example of $\widehat{T}$:
\[
(x_{t_1},\{r_1,r_3\},\{s_1,s_2\}), \; (x_{t_2},\{r_2,r_3\},\emptyset^1 ), \; \ldots, \; (x_{t_i},A_{t_i} ,B_{t_i}), \; \ldots ,
\]
where $A_{t_i} \subseteq \{r_1,r_2,r_3 \}$ and $B_{t_i} \subseteq \{s_1,s_2\}$. 
For any influence vector $g$, $(\widehat{T},g)$ is a featured time series.

\subsection{Graph representation of featured time series data}\label{subsec:Graph representation of featured time series data}

Consider a featured time series \((\widehat{T},g)\) as defined in Definition~\ref{def:Featured time series}.
A connected graph $G=(V,E,W_E)$ is constructed using $T$ as explained in Section~\ref{subsec:Frequency-based graph construction}.
The main goal of the proposed method is to generalize the weight function $W_E$ in $G$, yielding a more flexible and adaptive PH analysis by fully utilizing the given information in the time series. 
For such generalization, we use $T_f$ and $g$.

Consider two time series $T_f^0: \mathbb{T} \rightarrow \mathcal{P}(F^0)$ and $T_f^1: \mathbb{T} \rightarrow \mathcal{P}(F^1)$ such that 
$T_f(t) = (T_f^0(t),T_f^1(t))$ for any $t \in \mathbb{T}$.
Write $F^0 = \{r_1, \ldots, r_m \}$ and $V = \{ v_1, v_2, \ldots, v_{\lvert V \rvert} \}$ as ordered sets.
Define the \textit{zeroth count matrix} $C_0 = (c^0_{ij})$ using $F^0$ such that 
\begin{equation*}
c^0_{ij} \coloneqq
\begin{cases}
    \big\lvert \{t \in \mathbb{T} \mid v_i = T(t), r_j \in T_f^0(t) \} \big\rvert & \text{ if } j = 1, \ldots, m \\
    \big\lvert \{t \in \mathbb{T} \mid v_i = T(t), T_f^0(t) = \emptyset^0 \} \big\rvert & \text{ if } j = 0 \\
\end{cases},
\end{equation*}
where $\big\lvert \cdot \big\rvert$ denotes the cardinality of the set.
The $c^0_{ij}$ represents the total number of timestamps $t \in \mathbb{T}$ 
when the vertex $v_i$ in $V$ and feature $r_j$ in $F^0$ are encountered together in $\widehat{T}$.
The $i$th element of the first column of $C_0$ represents the total number of timestamps when $v_i$ appears without any zeroth features. 

Similarly, write $F^1 = \{s_1, \ldots, s_l \}$ and $E = \{ e_1, e_2, \ldots, e_{\lvert E \rvert} \}$ as ordered sets and define the \textit{first count matrix} $C_1 = (c^1_{ij})$ using $F^1$ such that 
\begin{equation*}
c^1_{ij} \coloneqq
\begin{cases}
    \big\lvert \{t_k \in \mathbb{T} \mid e_i = \{ T(t_k), T(t_{k+1}) \} , s_j \in T_f^1(t_k) \} \big\rvert &  \text{ if } j = 1, \ldots, l \\
    \big\lvert \{t_k \in \mathbb{T} \mid e_i = \{ T(t_k), T(t_{k+1}) \} , T_f^1(t_k) = \emptyset^1 \} \big\rvert &  \text{ if } j = 0 \\
\end{cases}.
\end{equation*}
The $c^1_{ij}$ represents the total number of timestamps $t \in \mathbb{T}$ 
when the edge $e_i$ and feature $s_j$ in $F^1$ are encountered together in $\widehat{T}$.
The $i$th element of the first column of $C_1$ represents the total number of the timestamps when $e_i$ appears without any first features. 

Set \(\vec{g_0}=(g(\emptyset^0),g(r_1),g(r_2),\ldots,g(r_m) )\) and \(\vec{g_1}=(g(\emptyset^1), g(s_1), g(s_2),\ldots, g(s_l))\) for an influence vector $g$. 
Consider two vectors $ C_0 \cdot \vec{g_0}$ and $ C_1 \cdot (\vec{g_1}+\vec{1})$.
Here, $\cdot$ denotes matrix multiplication, and $\vec{1}$ means the vector whose elements are all $1$.
The use of $\vec{1}$ in $C_1 \cdot (\vec{g_1}+\vec{1})$ is to ensure that, when $\vec{g_1} = 0$, the value $(C_1 \cdot \vec{1})_i$ aligns with the frequency $f_e$ defined in Section~\ref{subsec:Frequency-based graph construction} for edge $e$ in $E$.

Define a vertex weight function $\widehat{W}_V : V \rightarrow \mathbb{R}$ as $\widehat{W}_V({v_i}) = (C_0 \cdot \vec{g_0})_i$ for ${v_i} \in V$ and an edge weight function $\widehat{W}_E : E \rightarrow \mathbb{R}$ as $\widehat{W}_E(e_i) = ( C_1 \cdot (\vec{g_1}+\vec{1}))_i$ for $e_i \in E$.
Write $\widehat{W}_E(e_i)$ as $\widehat{f}_{e_i}$.
The weight $\widehat{f}_e$ is called \textit{the weighted frequency of an edge} $e$.
It is a generalization of the frequency explained in Section \ref{subsec:Frequency-based graph construction}. 
The weight $\widehat{W}_V(v)$ is called \textit{the weighted frequency of a vertex} $v$ associated with the zeroth features, $r_1, r_2, \ldots, r_m$.
When $\vec{g_0} = \vec{1}$ , the $\widehat{W}_V(v)$ is simply the frequency of appearance of vertex $v$ in $\widehat{T}$.

We then obtain the vertex- and edge-weighted graph $\widehat{G}^g=(V,E,\widehat{W}_V,\widehat{W}_E)$ from the featured time series $(\widehat{T},g)$.
These weight functions are used to define the metric space in the following section.

\subsection{Distance on vertex-edge weighted graphs}\label{subsec:Distance on vertex-edge weighted graphs}

Suppose that we have a graph $\widehat{G}^g=(V,E,\widehat{W}_V,\widehat{W}_E)$.
Let $\alpha = \min_{e \in E} (\widehat{W}_E(e))^{-1}$. 
Take an increasing function $\rho : \mathbb{R} \rightarrow (-1,1)$ into the open interval $(-1,1)$ such that $\rho(0)=0$.
We call $\rho$ an \textit{activation function}.
Define the \textit{length function} $L^g : E \rightarrow \mathbb{R}$ of the graph $\widehat{G}^g$  as 
\[
L^g(e) = (\widehat{W}_E(e))^{-1} - \alpha \left( \rho(\widehat{W}_V(a) + \widehat{W}_V(b)) \right) \text{ for any edge } e = \{a, b\} \in E.
\]
The distance $\widehat{d}$ of the graph $\widehat{G}^g$ is defined below.
\begin{definition} [Distance]\label{def:distance_weightedgraph}
Let \( v, w \in V \) be any vertices.
If \( v = w \), define \( \widehat{d}(v, w) = 0 \). Otherwise, the distance $\widehat{d}$ between $v$ and $w$ is defined as
\[ \widehat{d}(v,w) = \min\left\{ 
\sum_{\substack{e \in p}} L^g(e) \Big\lvert
\text{$p$ is a path in $G$ from $v$ to $w$}
\right\}.
\]
\end{definition}

The motivation for defining distance $\widehat{d}$ is as follows.
Consider an edge \(e = \{a,b\}\) in the graph $\widehat{G}^g$. 
As discussed in Section~\ref{subsec:Distance on weighted graphs}, we fundamentally use the reciprocal of $\widehat{W}_E$ for the distance $\widehat{d}$. 
The stronger the relevance between $a$ and $b$, the closer they should be positioned. 
In particular, compared with $W_E$, $\widehat{W}_E$ is a function that further assigns the sum of the influences on the features, meaning the more frequent the appearance of the first feature, which has a significant influence on the featured time series $\widetilde{T}$, the stronger the connection becomes.
Additionally, we used the vertex weight function $\widehat{W}_V$ for the distance $\widehat{d}$.
This is designed to reflect the influence of a single node $a$ on the other nodes in the entire network. 
The greater the vertex weight assigned to $a$, the shorter the distance from $a$ to all other nodes.
For example, consider a social network \(G\) in which node \(a\) is a famous influencer in the network and node $b$ is connected to $a$.
Even if there is no direct interaction between nodes \(a\) and \(b\), it is natural to consider the distance between them to be small if \(a\)'s influence is strong.
To achieve this, in the definition of the length function \(L^g\), we subtract the sum of the influences of nodes \(a\) and \(b\) from \(\widehat{W}_E(e)\), resulting in \(\widehat{W}_E(e) - (\widehat{W}_V(a) + \widehat{W}_V(b))\).
However, as $\widehat{W}_V$ and $\widehat{W}_E$ are independent functions, they can be negative. 
Therefore, we ultimately define the length \(L^g\) for any edge \(e = \{a, b\} \in E\) as 
\[
L^g(e) = (\widehat{W}_E(e))^{-1} - \alpha \left( \rho(\widehat{W}_V(a) + \widehat{W}_V(b)) \right).
\]
Figure~\ref{fig:vertex_distance_updated} presents a schematic illustration of how the graph $\widehat{G}^g=(V, E,\widehat{W}_V,\widehat{W}_E)$ changes as we consider more general vertex weights $\widehat{W}_V$ and edge weights $\widehat{W}_E$.

\begin{figure}[ht]
    \centering
    \includegraphics[width=0.95\textwidth]{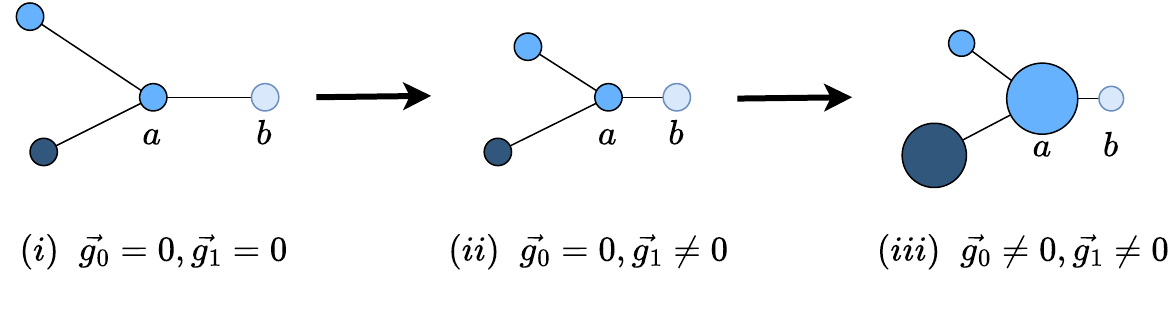}
    \label{fig:L_motive.drawio}
    \caption{
    Schematic illustrations of the weighted graph $\widehat{G}^g=(V, E, \widehat{W}_V, \widehat{W}_E)$ for various influence vectors $g$.
    (i) represents the case where $\widehat{W}_V = 0$ and $\widehat{W}_E = W_E$, (ii) shows the changes when only the edge weight is varied, and (iii) illustrates the changes in the vertex weights from those in (ii).
    }
    \label{fig:vertex_distance_updated}
\end{figure}

The distance $\widehat{d}$ satisfies the metric conditions.
\begin{proposition}\label{prop: d_ is metric}
Let \(\widehat{d}\) be the distance as defined in Definition~\ref{def:distance_weightedgraph}. Then, \(\widehat{d}\) is a metric.
\end{proposition}
\begin{proof}
Let \( v, w, z \in V \) be any vertices. 
If \( v = w \), we have \( \widehat{d}(v, w) = 0 \) by Definition~\ref{def:distance_weightedgraph}. 
In addition, the length function \(L^g\) is non-negative; therefore so \(\widehat{d}(v,w) = 0\) implies \(v = w\).
Suppose that \( v \neq w \). 
For any path \( p \) from \( v \) to \( w \), we can create a path \( p^{-1} \)  from \( w \) to \( v \) by simply reversing the order of the vertices in \( p \). 
Therefore, the distance $\widehat{d}$ satisfy \( \widehat{d}(v, w) = \widehat{d}(w, v) \). 
Finally, let us prove the condition of triangle inequality. 
Denote the path satisfying the minimality in the definition of \( \widehat{d}(v, z) \) as \( p_1 \), and the path in \( \widehat{d}(z, w) \) as \( p_2 \). 
Consider a path \( p_1 \cup p_2 \), which is obtained by concatenating \( p_1 \) and \( p_2 \).
Sum all values \( L^g(e) \) according to any edge $e$ in path $p_1 \cup p_2$, then it is $\widehat{d}(v, z) + \widehat{d}(z, w)$.
Because \( p_1 \cup p_2 \) is a path from \( v \) to \( w \), by the property of minimality in the definition of \( \widehat{d}(v, w) \), we have \( \widehat{d}(v, w) \leq \widehat{d}(v, z) + \widehat{d}(z, w) \). 
\end{proof}

\begin{corollary}
 $(V,\widehat{d})$ is a metric space for graph $\widehat{G}^g = (V,E,\widehat{W}_V,\widehat{W}_E)$
\end{corollary}

The distance \(\widehat{d}\) defined on \(\widehat{G}^g\) can be considered a generalization of the distance \(d\) from Definition~\ref{def:distance} from the perspective of the following proposition:
\begin{proposition}\label{prop: d and d_ are same}
  If the influence vector \( g \) is zero, then the two distances \( d \) in Definition~\ref{def:distance} and \( \widehat{d} \) in Definition~\ref{def:distance_weightedgraph} are identical on \( V \).
\end{proposition}
\begin{proof}
Because \(\vec{g_0}\) is zero, the vertex weight function \(\widehat{W}_V\) is zero. 
From \(\rho(0) = 0\), we get 
\[
    \sum_{\substack{ e \in p}} L^g(e)
    = \sum_{\substack{e = \{a,b\}\\ e \in p}}
    \left( (\widehat{W}_E(e))^{-1} - \alpha \rho(\widehat{W}_V(a) + \widehat{W}_V(b)) \right) 
    = \sum_{\substack{ e \in p}} (\widehat{W}_E(e))^{-1}.
\]
As \(\vec{g_1}\) is zero, it follows that \( \widehat{f}_{e_i} = (C_1 \cdot (\vec{g_1} + \vec{1}) )_i= (C_1 \cdot \vec{1})_i = f_e \) 
leading to \(\widehat{W}_E(e) = \widehat{f}_e = f_e = W_E(e)\).
Consequently, by the definition of \(\widehat{d}\), we obtain 
\begin{align*}
    \widehat{d}(v,w) 
    &= \min\left\{\sum_{\substack{ e \in p}} L^g(e) \bigg| p \text{ is a path between } v \text{ and } w\right\} \\
    &= \min\left\{\sum_{\substack{ e \in p}} ({W}_E(e))^{-1}  \bigg| p \text{ is a path between } v \text{ and } w\right\} = d(v,w).
\end{align*}
\end{proof}

\subsection{Activation function}\label{subsection:The auto activation function}
The role of \( \rho \) is to ensure that the values of the sum of any two vertex weights exist between $0$ and $1$ in the ascending order. 
If \( \rho \) is not appropriately chosen, all the elements of \( \rho(\widehat{W}_V(a) + \widehat{W}_V(b))   \) in Definition \ref{def:distance_weightedgraph} may become nearly identical for any edge $e=(a,b)$.
It is crucial to select \( \rho \) so that the distribution of the elements of \( \rho(\widehat{W}_V(a) + \widehat{W}_V(b)) \) facilitates effective analysis.
Although the choice of the activation function can vary depending on the problem, 
this process can also be efficiently automated.
We Consider the following activation function: 
$$
\rho(z) = 
\begin{cases} 
1 - e^{-z^2} & \text{if } z \geq 0, \\
0 & \text{if } z < 0.
\end{cases}
$$
Let \(M\) be the maximum value of \(\{ \widehat{W}_V(v_i) + \widehat{W}_V(v_j) \mid i \neq j \}\), where \(M \geq 0\).
Define \textit{the automatic activation function} \(\rho_*(z)\) as \(\rho_*(z) = \rho\left(\frac{2z}{M+1}\right)\).

\begin{proposition}\label{prop:auto_activation}
The function \(\rho_*(z)\) is Lipschitz continuous with a constant that is independent of the influence vector $g$, although \(M\) depends on \(\widehat{W}_V\).
\end{proposition}
\begin{proof}
We know that $\rho$ is Lipschitz continuous owing to the boundness of its derivative $\rho'$. There exists a constant $k$, independent of $\widehat{W}_V$, such that we have $\lvert \rho'(z) \rvert \leq k$ for any $z$. 
We can infer that $\lvert \rho_*'(z)  \rvert = \frac{2}{M+1} \lvert \rho'(\frac{2z}{M+1})\rvert \leq 2k$.
\end{proof}

In the stability theorem proved in Section~\ref{sec:Stability theorem for influence vectors}, 
the independence of the Lipschitz constant of the activation function from the influence vectors $g$ becomes crucial.
Proposition~\ref{prop:auto_activation} means that using the automatic activation function $\rho_*$ is permissible.

\subsection{Example: addressing information loss in PH calculation}\label{example:improve_infoloss}
Consider the time series $T$, as shown in Figure~\ref{fig:Time_temperature_nofeat}. 
We define the zeroth feature set \(F_0\) as \(\{ L, H \} \), representing humidity levels including low (L) and high (H), and the first feature set \(F_1\) as \(\{ |T(t+1) - T(t)| \; | \; t \in \mathbb{T} \} \), representing the magnitude of the temperature changes. 
As shown in Figure~\ref{fig:Time_temperature_feat}, $T$ is complemented by a tuple of features \( (f_1,f_2) \in F_0 \times F_1 \), providing additional information.
This is denoted as \(\widehat{T}\).

To emphasize the sudden changes in temperature, define the influence vector \(g\) such that \(g(4) = 5\) for $4 \in F_1$ and \(g(x) = 0\) for all other features of \(x\). 
In other words, \(\vec{g}_0\) and \(\vec{g}_1\) are expressed as

\noindent
\hfill
\begin{tikzpicture}[scale=2]
    \matrix [matrix of math nodes,style={text depth=1ex,text height=1ex,text width=1em},left delimiter=(,right delimiter=)^T] (g0) at (0,0)
    {
    0 \pgfmatrixnextcell 0 \pgfmatrixnextcell 0 \\
    };
      
    \node[left=1em of g0-1-1] (row1) {$\vec{g_0} = $};
    \node[above=0.5em of g0-1-1] (row1) {$\emptyset^0$};
    \node[above=0.5em of g0-1-2] (row2) {$H$};
    \node[above=0.5em of g0-1-3] (row2) {$L$};

    \matrix [matrix of math nodes,style={text depth=1ex,text height=1ex,text width=1em},left delimiter=(,right delimiter=)^T] (g1) at (2.5,0)
    {
    0 \pgfmatrixnextcell 0 \pgfmatrixnextcell 5 \\
    };
      
    \node[left=1em of g1-1-1] (row1) {$\vec{g_1} = $};
    \node[above=0.5em of g1-1-1] (row1) {$\emptyset^1$};
    \node[above=0.5em of g1-1-2] (row2) {$1$};
    \node[above=0.5em of g1-1-3] (row1) {$4$};

\node at ($(g0)!0.41!(g1)$) (and) { and };
\end{tikzpicture}.
\hfill
\hspace*{0pt}

\begin{figure}[ht]
    \centering
        \includegraphics[width=\linewidth]{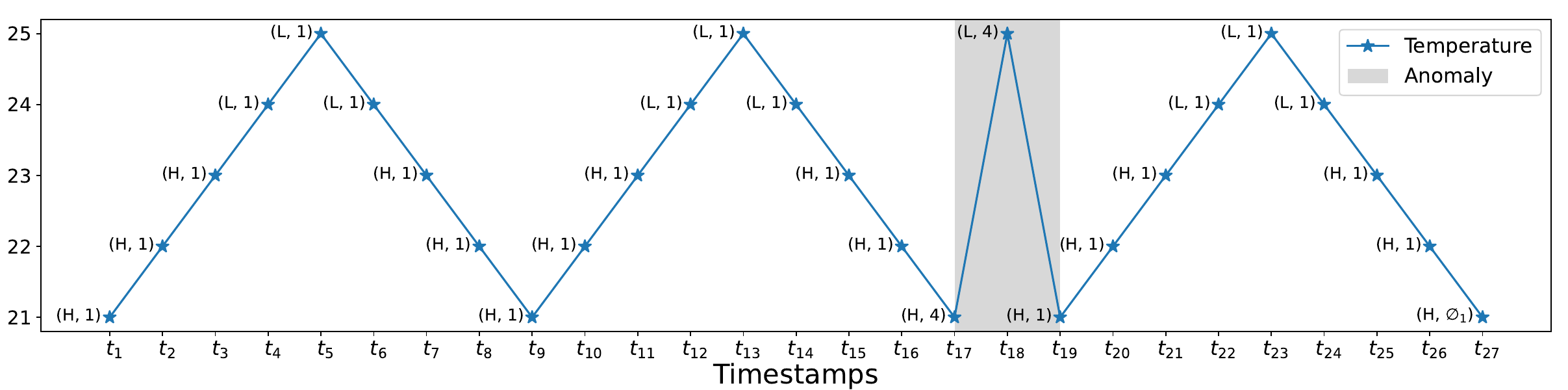}
        \caption{A time series $\widehat{T}$ with features  added from the original time series $T$ shown in Figure~\ref{fig:Time_temperature_nofeat}.}
        \label{fig:Time_temperature_feat}
\end{figure}

For the featured time series \((\widehat{T}, g)\), the vertex weight function  \(\widehat{W}_V\) is zero, as \(\vec{g_0} = \vec{0}\).
The edge weight function \(\widehat{W}_E\) can be obtained using the first count matrix \(C_1\) as follows:

\hspace*{0pt}
\hfill
\begin{tikzpicture}[scale=1.5]
 \matrix [matrix of math nodes,style={text depth=1ex,text height=1ex,text width=1em},left delimiter=(,right delimiter=)] (C1)
 {
  0 \pgfmatrixnextcell 6 \pgfmatrixnextcell 0 \pgfmatrixnextcell \\
  0 \pgfmatrixnextcell 6 \pgfmatrixnextcell 0 \pgfmatrixnextcell \\
  0 \pgfmatrixnextcell 6 \pgfmatrixnextcell 0 \pgfmatrixnextcell \\
  0 \pgfmatrixnextcell 6 \pgfmatrixnextcell 0 \pgfmatrixnextcell \\
  0 \pgfmatrixnextcell 0 \pgfmatrixnextcell 2 \pgfmatrixnextcell \\
 };
  
\node[left=1em of C1-1-1] (row1) {$(21,22)$};
\node[left=1em of C1-2-1] (row2) {$(22,23)$};
\node[left=1em of C1-3-1] (row1) {$(23,24)$};
\node[left=1em of C1-4-1] (row2) {$(24,25)$};
\node[left=1em of C1-5-1] (row2) {$(21,25)$};

\node[above=0.5em of C1-1-1] (row1) {$\emptyset^1$};
\node[above=0.5em of C1-1-2] (row2) {$1$};
\node[above=0.5em of C1-1-3] (row1) {$4$};
  \matrix [matrix of math nodes,style={text depth=1ex,text height=1ex,text width=1em},left delimiter=(,right delimiter=),label=above:\( \vec{g_1}+\vec{1} \)] (v) at (2,0)
  {
    1 \\
    1 \\
    6 \\
  };
  \matrix [matrix of math nodes,style={text depth=1ex,text height=1ex,text width=1em},left delimiter=(,right delimiter=),label=above:\( C_1 \cdot (\vec{g_1}+\vec{1}) \)] (Cv) at (4,0)
  {
    6 \\
    6 \\
    6 \\
    6 \\
    12 \\
  };

\node at ($(C1)!0.6!(v)$) (times) {$\times$};
\node at ($(v)!0.5!(Cv)$) (equal) {$=$};
\end{tikzpicture}.
\hfill
\hspace*{0pt}

\begin{figure}[ht]
    \centering
    \subfloat[Rips filtration when $\vec{g}_0 = (0,0,0)$ and $\vec{g}_1 = (0,0,5)$. \label{fig:Rips_filtration2}]{
        \includegraphics[width=0.95\textwidth]{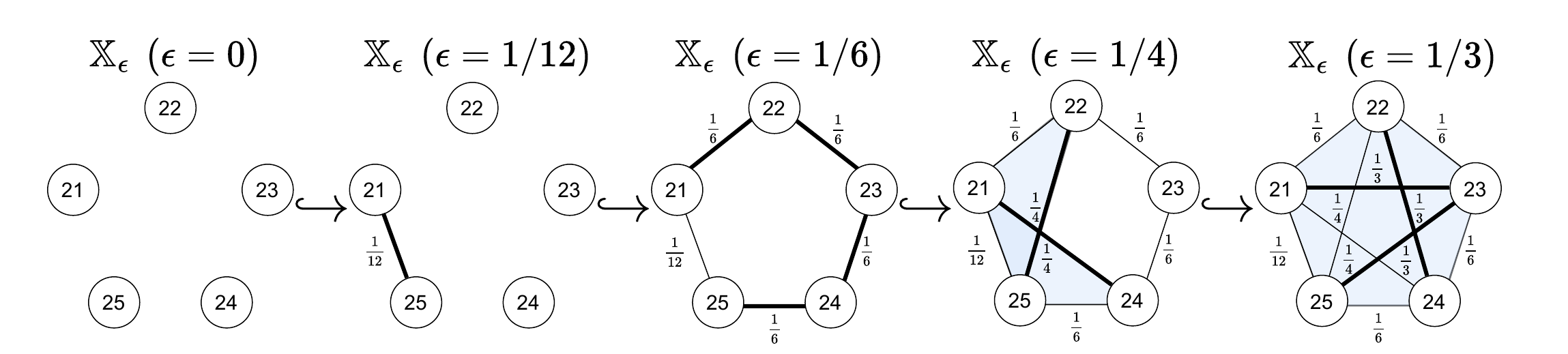}
    }\\ 
    \subfloat[PD when $\vec{g}_0 = (0,0,0)$ and $\vec{g}_1 = (0,0,0)$. \label{fig:PD_[0, 0, 0]_[0, 0, 0]}]{
        \includegraphics[width=0.45\linewidth]{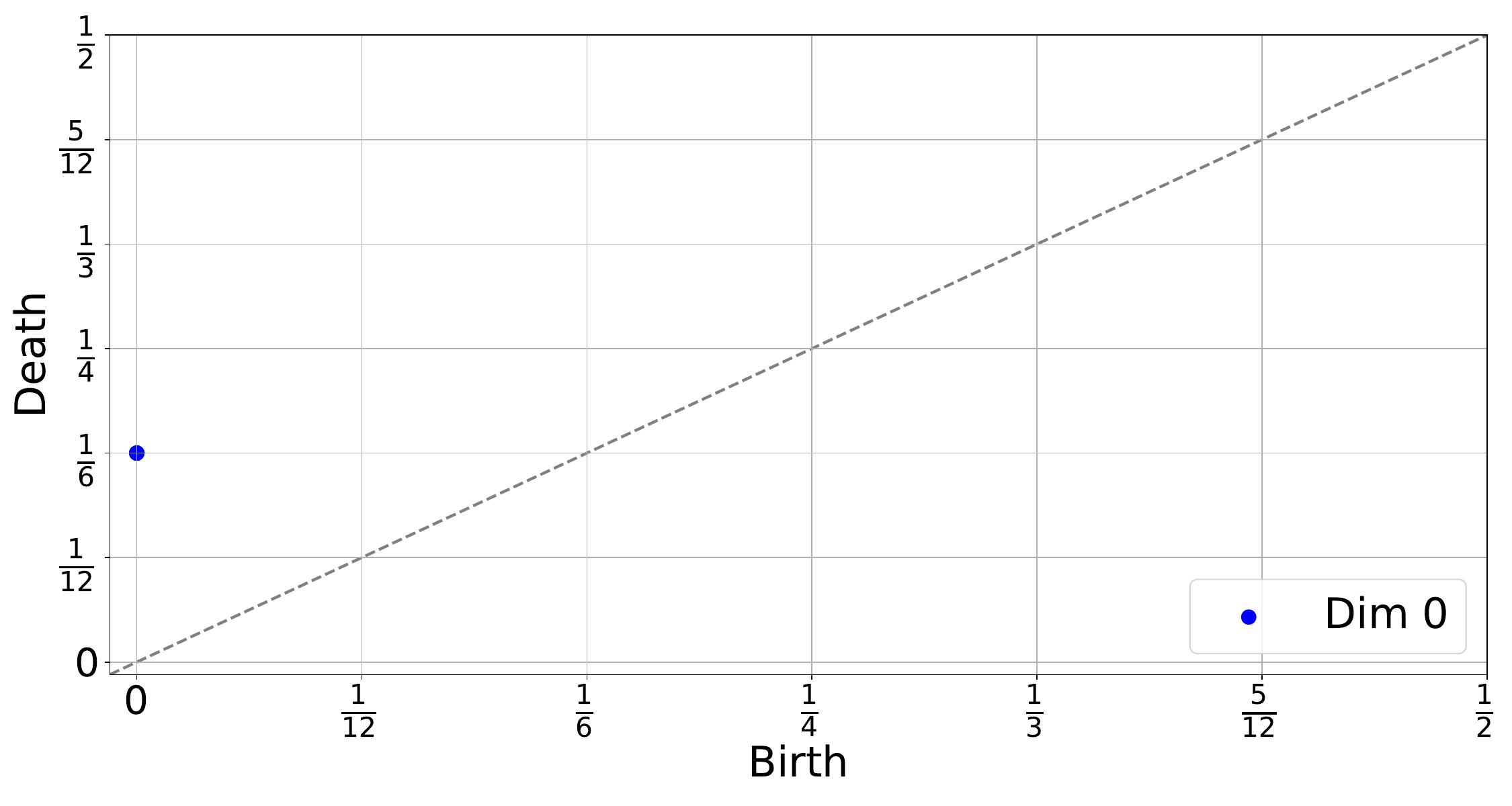}
    }
    \hspace{2mm}
    \subfloat[PD when $\vec{g}_0 = (0,0,0)$ and $\vec{g}_1 = (0,0,5)$. \label{fig:PD_[0, 0, 0]_[0, 0, 5]}]{
        \includegraphics[width=0.45\linewidth]{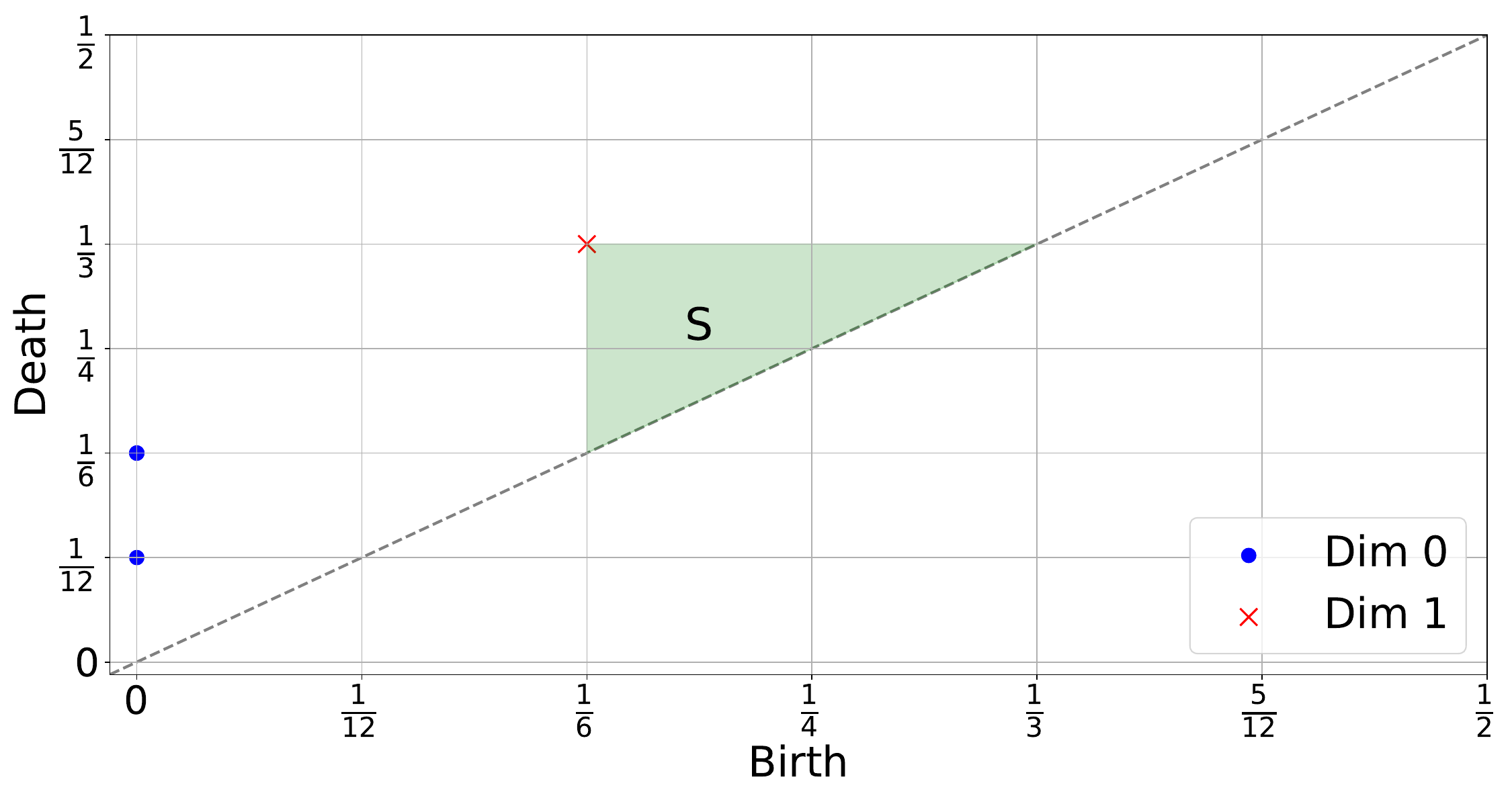}
    }\\
    \subfloat[All points of the PDs for $\vec{g}_0 = (0, i, j)$ with $i, j \in \mathcal{I} = \left\{ \frac{n}{100} \mid 0 \leq n \leq 500, n \in \mathbb{N} \right\}$ when $\vec{g}_1 = (0,0,0)$. \label{fig:PD_vary_[0,0,0]}]{
        \includegraphics[width=0.45\linewidth]{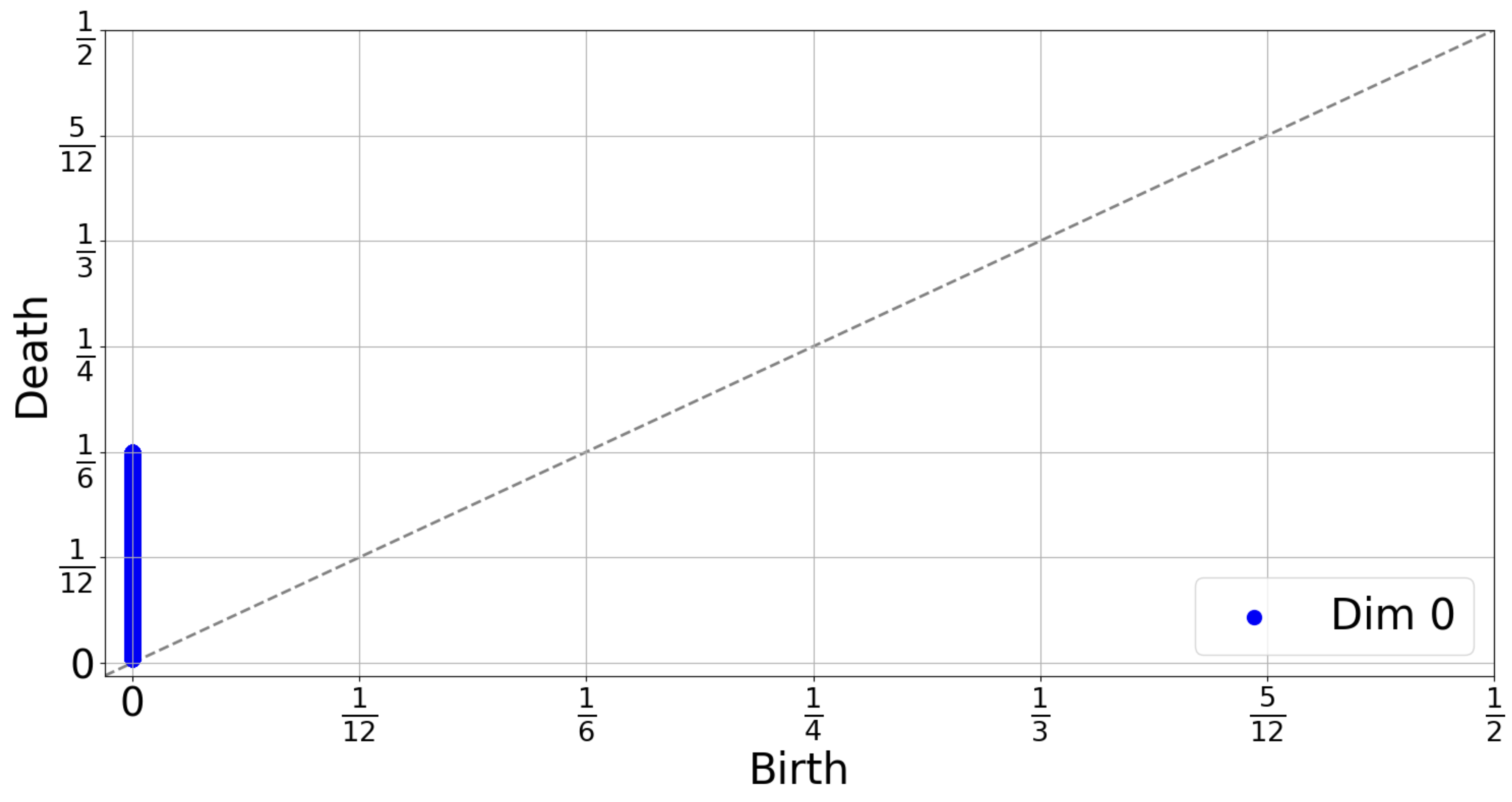}
    }
    \hspace{2mm}
    \subfloat[All points of the PDs for $\vec{g}_0 = (0, i, j)$ with $i, j \in \mathcal{I} = \left\{ \frac{n}{100} \mid 0 \leq n \leq 500, n \mathbb{N} \right\}$ when $\vec{g}_1 = (0,0,5)$. \label{fig:PD_vary_[0,0,5]}]{
        \includegraphics[width=0.45\linewidth]{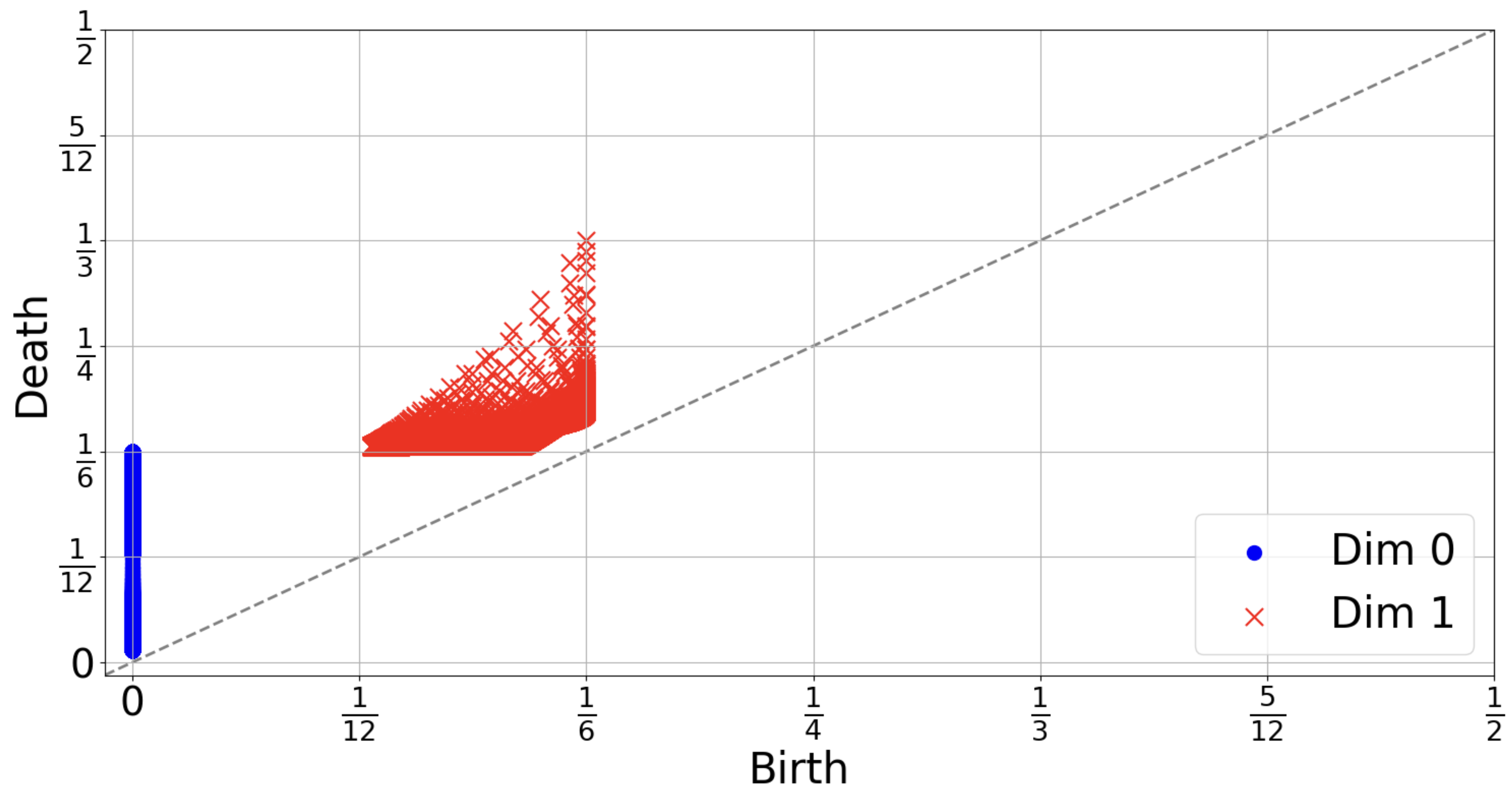}
    }
    \caption{Rips filtration and persistence diagrams (PDs) for various influence vectors for $\widehat{G}^g$. }
    \label{fig:PH_diagram2}
\end{figure}

Figure.~\ref{fig:Rips_filtration2} illustrates the Rips filtration for the metric space \((V, \widehat{d})\), which is constructed from the graph \(\widehat{G}^g = (V, E, \widehat{W}_V, \widehat{W}_E)\). In the filtration, the edge \(\{21,25\}\) first appears in \(\mathbb{X}_{\epsilon}\) when \(\epsilon = d(21,25) = \frac{1}{12}\). 
Unlike the example discussed in Section~\ref{subsec:Discussion of Information loss through an example}, by assigning an influence value of \(5\) to \(4 \in F_1\), we emphasize the importance of the edge $\{21,25\}$. 
This not only results in the formation of a point in the \(1\)-dimensional PD, as observed in Figure~\ref{fig:Rips_filtration2}, but also leads to the emergence of a point in the \(0\)-dimensional PD with a death time of \(d(21,25) = \frac{1}{2}\). 
These developments are illustrated in Figures~\ref{fig:Rips_filtration2} and ~\ref{fig:PD_[0, 0, 0]_[0, 0, 5]}, which contrast with the scenario in Figures~\ref{fig:Rips_complex} and ~\ref{fig:PD_[0, 0, 0]_[0, 0, 0]}, where the latter displays a single point representing four \(0\)-dimensional points.
This example illustrates that by adjusting the influence vector $g$, previously unrepresented domain-specific information can now be effectively preserved in the computation of PH.

In this example, the humidity was used as the zeroth feature. 
Figures~\ref{fig:PD_vary_[0,0,0]} and ~\ref{fig:PD_vary_[0,0,5]} show that the influence on the zeroth features affects the persistence (death time \(-\) birth time) of each point rather than the existence of \(1\)-dimensional points. 
Although not extensively covered in this example, the zeroth feature is considered significant in the stock data, as discussed in Section~\ref{subsec:The first method}.
Furthermore, in Section~\ref{subsec:The first method}, the area of the shaded region \(S\) in Figure~\ref{fig:PD_[0, 0, 0]_[0, 0, 5]} is used for anomaly detection on the stock data.

\subsection{Feature selection}\label{subsec:Feature selection}
Feature selection is problem-dependent, and features should be properly categorized for analysis. 
Suppose that we have a time series \( T \) and a collection \( \mathcal{O} \) of objects, including the various domain knowledge that we want to consider as features.
To classify \( \mathcal{O} \) into the zeroth and first features, it is important to understand that the $i$th features of the featured time series data determine the weights of \(i\)-simplices in graphs.
The vertex set of graphs consists of single observations $\{T(t)\}$, making it sensible to associate any feature \( s_0 \in \mathcal{O} \) closely related to a single observation with the zeroth feature.
In addition, because the first features influence the weights of edges \( \{T(t), T(t+1)\} \) for some $t \in \mathbb{T}$, if any feature \( s_1 \in \mathcal{O} \) relates to any two observations simultaneously, it would be proper to regard it as the first feature. 

In Example~\ref{example:improve_infoloss}, humidity is considered the zeroth feature because it pertains to the information from a single observation \( T(t) \) at each timestamp \( t \in \mathbb{T} \).
Furthermore, the magnitude of the temperature change is classified as the first feature because it is derived from two consecutive observations \( T(t) \) and \( T(t+1) \) at each timestamp \( t \in \mathbb{T} \).
\section{Stability theorem for influence vectors}\label{sec:Stability theorem for influence vectors}

The stability theorem for persistence diagrams is crucial for ensuring reliable topological inferences.

\begin{lemma}[Cohen-Steiner et al.~\cite{cohen2005stability}]\label{lemma:Stability theorem for PH}
Let $\mathbb{X}$ be a finite abstract simplicial complex.
For any two functions $h_1, h_2 : \mathbb{X} \rightarrow \mathbb{R}$, 
the persistence diagrams $\text{dgm}_p(h_1)$ and $\text{dgm}_p(h_2)$ satisfy
\[
D_B(\text{dgm}_p(h_1), \text{dgm}_p(h_2)) \leq ||h_1 - h_2||_{\infty} 
\]
for any dimension $p$, where $D_B$ is the bottleneck distance.
\end{lemma}

If the change in the function $h: \mathbb{X} \rightarrow \mathbb{R}$ is small, then the change in the PD is also tiny.
We prove that the PD remains stable when adjusting the influence vector $g$ for the featured time series $(\widehat{T},g)$.
This section proves that the PD remains stable when adjusting the influence vector $g$ for the featured time series $(\widehat{T},g)$.
\begin{theorem}[Stability of the proposed method]\label{Main_Thm}
Let \(\widehat{T}\) be a time series with augmented feature information derived from a times series \(T\).
If the activation function $\rho$ is Lipschitz continuous, then there exists a constant $C=C(\widehat{T},\rho) > 0$ such that for any two influence vectors \(g\) and \(g'\), the persistence diagrams $\text{dgm}_p(g)$ and $\text{dgm}_p(g')$ satisfy
\[
 D_B(\text{dgm}_p(g), \text{dgm}_p(g')) \leq C \lVert g - g' \rVert_{\infty}
\]
for any dimension $p$, where $D_B$ is the bottleneck distance. 
\end{theorem}

The function $\Phi: g \mapsto \text{dgm}_p(g)$ is Lipschitz continuous.
For the count matrices  \( C_0 = (c^0_{ij}) \) and  \( C_1 = (c^1_{ij}) \), 
define \( C^{\text{max}}_0 \) and \( C^{\text{min}}_1 \), respectively, as follows:
\[
 C^{\text{max}}_0 := \max_{v \in V} \sum_{j} c^0_{i_v j} \;\;\text{ and }\;\; C^{\text{min}}_1 := \min_{e \in E} \sum_{j} c^1_{i_e j}.
\]
By the Lipschitz continuity of the  activation function \( \rho \), 
we can find a constant \( k \) such that 
\[
\left| \rho(x) - \rho(y) \right| \leq k \left| x - y \right| \text{ for any } x,y.
\]
The constant \( C \) of the main theorem is 
\[
C = \left( k C^{\text{max}}_0 + 1 \right) \frac{2}{C^{\text{min}}_1} (\lvert V \rvert-1).
\]

\subsection{Proof of the stability theorem for influence vectors}
Consider any featured time series $(\widehat{T},g)$.
We append $g$ to all notations if they are variables dependent on $g$. 
Consider the connected graph $\widehat{G}^g = (V, E, \widehat{W}^g_V, \widehat{W}^g_E)$ and the metric space \((V, \widehat{d}^{g})\).
Recall that in Section \ref{subsec:Persistent homology of metric spaces}, we defined the filtration value function $h_g: \mathbb{X} \rightarrow \mathbb{R}$, where $\mathbb{X} = \mathcal{P}(V)$. By Lemma \ref{lemma:Stability theorem for PH}, we have $D_B \left( \text{dgm}_p(g), \text{dgm}_p(g') \right) \leq \lVert h_g - h_{g'} \rVert_{\infty}$ for any two influence vectors $g$ and $g'$. 
Therefore, it suffices to show that there exists a constant $C$ independent of $g$ and $g'$, satisfying $\lVert h_g - h_{g'} \rVert_{\infty} \leq C \lVert g - g' \rVert_{\infty}$.

Since $\mathbb{X}$ is finite, we have $\lVert h_{g} - h_{g'} \rVert_{\infty} = \lvert (h_{g} - h_{g'})(\sigma_*) \rvert$ for a $p$-simplex $\sigma_* \in \mathbb{X}$. By the maximality of $h_{g}$ and $h_{g'}$, $h_{g}(\sigma_*) = \widehat{d}^{g}(v_{g},w_{g})$ and $h_{g'}(\sigma_*) = \widehat{d}^{g'}(v_{g'},w_{g'})$ for edges $(v_{g},w_{g}), (v_{g'},w_{g'}) \subseteq \sigma_*$. Consider an interpolation map $q(t) = (1-t)g + tg'$ for a real number $t$, where $q(0)=g$ and $q(1)=g'$. 
Define a function $H(t) = \widehat{d}^{q(t)}(v_{g},w_{g}) - \widehat{d}^{q(t)}(v_{g'},w_{g'})$, where $\widehat{d}^{q(t)}$ is the distance in the graph $\widehat{G}^{q(t)}$, as defined in Definition~\ref{def:distance_weightedgraph}. 
By the maximality of the functions $h_g$ and $h_{g'}$, we have $H(0) \ge 0$ and $H(1) \le 0$. 
Because $H$ is continuous on the closed interval $[0,1]$, by the intermediate value theorem, there exists a real number $t_0 \in [0,1]$ such that $H(t_0)=0$.
If we denote $q(t_0)$ as $g^*$, we have 
\begin{align*}
\lvert (h_{g} - h_{g'})(\sigma_*) \rvert &= \lvert \widehat{d}^{g}(v_{g},w_{g}) - \widehat{d}^{g'}(v_{g'},w_{g'}) \rvert \\
&\leq \lvert \widehat{d}^{g}(v_{g},w_{g}) - \widehat{d}^{g^*}(v_{g},w_{g}) \rvert + \lvert \widehat{d}^{g^*}(v_{g'},w_{g'}) - \widehat{d}^{g'}(v_{g'},w_{g'}) \rvert \\
&= \lvert h_{g} - h_{g^*} \rvert((v_{g},w_{g})) + \lvert h_{g^*} - h_{g'} \rvert((v_{g'},w_{g'})).
\end{align*}
If we can prove there exists a constant $C$ such that $\lvert h_{g_1} - h_{g_2} \rvert(e) \leq C \lVert g_1 - g_2 \rVert_{\infty}$ for any influence vectors $g_1, g_2$, and edge $e \in E$, then we have
\begin{align*}
\lvert h_{g} - h_{g^*} \rvert((v_{g},w_{g})) + \lvert h_{g^*} - h_{g'} \rvert((v_{g'},w_{g'}))
&\leq C \lVert g - g^* \rVert_{\infty} + C \lVert g' - g^* \rVert_{\infty} \\
&= C t_0 \lVert g - g' \rVert_{\infty}  + C(1- t_0) \lVert g - g' \rVert_{\infty} \\
&= C \lVert g - g' \rVert_{\infty},
\end{align*}
which completes the proof. Therefore, it suffices to show that for any influence vectors $g$ and $g'$, and edge $e \in E$, there exists a constant $C$ such that
\(
\lvert h_{g} - h_{g'} \rvert(e) \leq C \lVert g - g' \rVert_{\infty}.
\)

Take any influence vectors \(g\), \(g'\), and edge \(e_*=(v,w) \in E\).
Consider an interpolation map \(q(t) = (1-t)g + tg'\) for the real number \(t\) such that \(q(0)=g\) and \(q(1)=g'\).
For any influence vector $q(t)$, we have 
\(
\widehat{d}^{q(t)}(v,w) = \sum_{e \in p^{q(t)}} L^{q(t)}(e)
\)
for some path \(p^{q(t)}\) from \(v\) to \(w\), as described in Definition~\ref{def:distance_weightedgraph}.
Define a function 
\(
K(t)=\sum_{e \in p^g}L^{q(t)}(e) - \sum_{e \in p^{g'}}L^{q(t)}(e)
\).
By the minimality of paths \(p^g\) and \(p^{g'}\), we deduce \(K(0) \leq 0\) and \(K(1) \geq 0\). 
By the intermediate value theorem, there exists \(t_1\in [0,1]\) such that \(K(t_1) = 0\).
If we denote $q(t_1)$ as $g^*$, then 
\begin{align}
\lvert (h_g - h_{g'})(e_*) \rvert   = & \Big\lvert \widehat{d}^{g}(v,w) - \widehat{d}^{g'}(v,w) \Big\rvert = \Big\lvert \sum_{e \in p^g}L^g(e) - \sum_{e \in p^{g'}}L^{g'}(e) \Big\rvert \nonumber\\
 = & \Big\lvert \sum_{e \in p^g}L^g(e) -\sum_{e \in p^g}L^{g*}(e) + \sum_{e \in p^{g'}}L^{g*}(e) - \sum_{e \in p^{g'}}L^{g'}(e) \Big\rvert \nonumber\\
 \leq & \sum_{e \in p^g}\lvert L^g(e) - L^{g*}(e) \rvert + \sum_{e \in p^{g'}}\lvert L^{g*}(e) - L^{g'}(e) \rvert. \label{eq:total}
\end{align}
Consider the first summation in inequality (\ref{eq:total}) as follows:
\begin{align*}
&  \sum_{e \in p^g}\left\lvert L^g(e) - L^{g*}(e) \right\rvert  \\
& = \sum_{e \in p^g}\Big\lvert (\widehat{W}^g_E(e))^{-1} -(\widehat{W}^{g^*}_E(e))^{-1} - \alpha^g( \rho(\widehat{W}^g_V(a) + \widehat{W}^g_V(b)) + \alpha^{g^*}( \rho(\widehat{W}^{g^*}_V(a) + \widehat{W}^{g^*}_V(b)) \Big\rvert \\
& \leq \sum_{e \in p^g}\Big\lvert (\widehat{W}^g_E(e))^{-1} -(\widehat{W}^{g^*}_E(e))^{-1} \Big\rvert + \Big\lvert \alpha^g( \rho(\widehat{W}^g_V(a) + \widehat{W}^g_V(b)) - \alpha^{g^*}( \rho(\widehat{W}^{g^*}_V(a) + \widehat{W}^{g^*}_V(b)) \Big\rvert.
\end{align*}
We can express $(\widehat{W}^g_E(e))^{-1} = 1 / \left( \sum_j (g_1(s_j)+1)c^1_{i_e j} \right)$, where $i_e$ designates the position of $e$ in the rows of the first count matrix \(C_1 = (c^1_{ij})\). Observe that
\begin{align}
\sum_{e \in p^g} \Big\lvert (\widehat{W}^g_E(e))^{-1}-(\widehat{W}^{g^*}_E(e))^{-1} \Big\rvert = &
\sum_{e \in p^g} \left\lvert \frac{1}{\sum_j (g_1(s_j)+1)c^1_{i_e j}} - \frac{1}{\sum_j (g^*_1(s_j)+1)c^1_{i_e j}} \right\rvert \nonumber\\
= & \sum_{e \in p^g} \left\lvert \frac{\sum_j c^1_{i_e j}(g_1(s_j)-g^*_1(s_j))}{(\sum_j (g_1(s_j)+1)c^1_{i_e j})(\sum_j (g^*_1(s_j)+1)c^1_{i_e j})}\right\rvert \nonumber\\
\leq & \sum_{e \in p^g} \left\lvert \frac{\sum_j c^1_{i_e j}(g_1(s_j)-g^*_1(s_j))}{(\sum_j c^1_{i_e j})(\sum_j c^1_{i_e j})}\right\rvert \nonumber\\
\leq & \sum_{e \in p^g} \left\lvert \frac{(\sum_j c^1_{i_e j})}{(\sum_j c^1_{i_e j})(\sum_j c^1_{i_e j})}\right\rvert \lVert g - g^* \rVert_{\infty} \nonumber\\
\leq & t_1 \sum_{e \in p^g} \left(\frac{1}{(\sum_j c^1_{i_e j})}\right) \lVert g - g' \rVert_{\infty} \hspace{3mm} \text{ by }g-g^*=t_1(g-g') \nonumber\\
\leq & t_1 \sum_{e \in p^g} \left(\frac{1}{(\sum_j c^1_{i_e j})}\right) \lVert g - g' \rVert_{\infty} \nonumber\\
\leq &  t_1 \sum_{e \in p^g} \left(\frac{1}{C^{\text{min}}_1}\right) \lVert g - g' \rVert_{\infty} \hspace{3mm} \text{ by } C^{\text{min}}_1 = \min_{e \in E} \sum_j c^1_{i_e j} \nonumber\\
\leq &  t_1 \frac{1}{C^{\text{min}}_1} \lVert g - g' \rVert_{\infty} (\lvert V \rvert-1). \label{eq:1_part}
\end{align}
In the last inequality (\ref{eq:1_part}), we utilize the fact that the number of edges in a path $p$, without repeated vertices, is precisely $\lvert V_p \rvert - 1$, where $\lvert V_p \rvert$ denotes the number of vertices in $p$. 
The path $p^g$ does not contain any repeated vertices owing to its minimality in the definition of the distance $d$.

We demonstrate that $\alpha^g$ is bounded for any given $g$; recall that $\alpha^g = \min_{e \in E} (\widehat{W}^g_E(e))^{-1}$. Given that all elements of an influence vector are non-negative,
\[
\sum_j (g_1(s_j)+1)c^1_{i_ej} \geq \sum_j c^1_{i_ej} \geq \text{min}_{e \in E}\{ \sum_j c^1_{i_ej} \} = C^{\text{min}}_1
\]
holds for any edge $e \in E$. Therefore, we can assert that $\alpha^g \leq \frac{1}{C^{\text{min}}_1}$ for any $g$.

\begin{lemma} Given any pair of influence vectors $g$ and $g'$, it holds true that $\lvert \alpha^g - \alpha^{g'} \rvert \leq \frac{1}{C^{\text{min}}_1} \lVert g - g' \rVert_{\infty}$ .
\end{lemma}

\begin{proof}
Observe that for any edge $e \in E$,
\begin{align}
\lvert (\widehat{W}^g_E(e))^{-1} - (\widehat{W}^{g'}_E(e))^{-1} \rvert & = \left\vert \frac{\sum_j c^1_{i_ej}(g_1-g'_1)(t_j)}{\left( \sum_j (g_1(s_j)+1) c^1_{i_ej}\right) \left( \sum_j (g'_1(t_j)+1) c^1_{i_ej}\right)} \right\rvert \nonumber \\ 
& \leq \left\vert \frac{\left( \sum_j c^1_{i_ej} \right) \lVert g-g'\rVert_{\infty} } {\left( \sum_j  c^1_{i_ej}\right) \left( \sum_j c^1_{i_ej}\right)} \right\rvert
= \frac{ 1 } { \left( \sum_j c^1_{i_ej}\right)} \lVert g-g'\rVert_{\infty} \nonumber\\
& \leq \frac{1}{C^{\text{min}}_1}\lVert g-g'\rVert_{\infty}.\label{eq:Lipschit alpha^g}
\end{align}

Assume that $\lvert \alpha^g - \alpha^{g'}\rvert > \frac{1}{C^{\text{min}}_1}\lVert g-g'\rVert_{\infty}$.
Without a loss of generality, suppose $\alpha^g > \alpha^{g'}$.
Because the set of edges $E$ is finite, there exist edges $e_1, e_2 \in E$ such that $\alpha^g=(\widehat{W}^g_E(e_1))^{-1}$ and $\alpha^{g'}=(\widehat{W}^{g'}_E(e_2))^{-1}$.
From the inequality \ref{eq:Lipschit alpha^g}, we derive that $\lvert \alpha^{g'} - (\widehat{W}^g_E(e_2))^{-1} \rvert \leq \frac{1}{C^{\text{min}}_1} \lVert g-g' \rVert_{\infty}$, leading to $(\widehat{W}^g_E(e_2))^{-1} \leq \alpha^{g'} + \frac{1}{C^{\text{min}}_1} \lVert g-g' \rVert_{\infty} $.
However, by assumption, we obtain $\alpha^{g'} + \frac{1}{C^{\text{min}}_1} \lVert g-g' \rVert_{\infty} < \alpha^g$. This implies $(\widehat{W}^g_E(e_2))^{-1} < \alpha^g$, which contradicts the definition of $\alpha^g$ as the minimum value across all edges $e \in E$ for a given $g$.    
\end{proof}

For any edge $e = (a,b) $, if we write $\rho(\widehat{W}^g_V(a) + \widehat{W}^g_V(b))$ as $D(g,e)$, then 
\[
 \sum_{e \in p^g} \Big\lvert \alpha^g( \rho(\widehat{W}^g_V(a) + \widehat{W}^g_V(b)) - \alpha^{g^*}( \rho(\widehat{W}^{g^*}_V(a) + \widehat{W}^{g^*}_V(b)) \Big\rvert  = \sum_{e \in p^g} \Big\lvert \alpha^gD(g,e)  - \alpha^{g^*}D(g^*,e)  \Big\rvert.
\]
The vertex weight $\widehat{W}^g_V(v)$ can be expressed as $\widehat{W}^g_V(v) = \sum_j g_0(r_j)c^0_{i_v j}$, where $i_v$ designates the position of vertex $v$ in the rows of the zeroth count matrix $C_0 = (c^0_{ij})$.
Because $\rho$ is Lipschitz continuous, there exists a positive constant $k$ such that $\lvert \rho(z) - \rho(z')\rvert \leq k \lvert z - z' \rvert $ for any two real numbers $z$ and $z'$.
Hence, 
\begin{align*}
    \left\lvert  D(g,e) -D(g^*,e) \right\rvert & = \left\lvert \rho(\widehat{W}^g_V(a) + \widehat{W}^g_V(b)) - \rho(\widehat{W}^{g^*}_V(a) + \widehat{W}^{g^*}_V(b))  \right\rvert \\
    & \leq k \left\lvert (\widehat{W}^g_V(a) +\widehat{W}^g_V(b)) - (\widehat{W}^{g^*}_V(a) + \widehat{W}^{g^*}_V(b)) \right\rvert \\
    & = k \left\lvert \left(\sum_{j}c^0_{i_a j} (g_0-g^*_0)(s_j) + \sum_{j}c^0_{i_b j} (g_0-g^*_0)(s_j)\right) \right\rvert  \\
    & \leq k \lVert g - g^* \rVert_{\infty} \sum_{j} (c^0_{i_a j} + c^0_{i_b j}) \\
    & \leq 2 k C^{\text{max}}_0  \lVert g - g^*\rVert_{\infty} \;\; \text{ by } C^{\text{max}}_0 = \max_{v \in V} \sum_{j} c^0_{i_v j}. 
\end{align*}
Therefore, we have 
\begin{align}
\sum_{e \in p^g} \Big\lvert \alpha^gD(g,e)  - \alpha^{g^*}D(g^*,e)  \Big\rvert   & 
= \sum_{e \in p^g} \Big\lvert \alpha^gD(g,e) - \alpha^gD(g^*,e) + \alpha^gD(g^*,e)- \alpha^{g^*}D(g^*,e)  \Big\rvert \nonumber\\
& \leq \sum_{e \in p^g}\lvert \alpha^g \rvert \left\lvert  D(g,e) -D(g^*,e) \right\rvert + \sum_{e \in p^g}\lvert \alpha^g - \alpha^{g^*}\rvert \left\lvert D(g^*,e) \right\rvert \nonumber\\
& \leq \sum_{e \in p^g}\frac{1}{C^{\text{min}}_1}\left\lvert  D(g,e) -D(g^*,e) \right\rvert +  \sum_{e \in p^g}\lvert \alpha^g - \alpha^{g^*}\rvert  \nonumber\\
& \leq \sum_{e \in p^g}\frac{2}{C^{\text{min}}_1}k C^{\text{max}}_0 \lVert g-g^*\rVert_{\infty} +  \sum_{e \in p^g}\frac{1}{C^{\text{min}}_1}\lVert g-g^*\rVert_{\infty} \nonumber\\
& = (2kC^{\text{max}}_0 + 1)\frac{\lVert g-g^*\rVert_{\infty}}{C^{\text{min}}_1}  \sum_{e \in p^g} 1 \nonumber\\
& \leq (2kC^{\text{max}}_0 + 1)\frac{\lVert g-g^* \rVert_{\infty}}{C^{\text{min}}_1}(\lvert V \rvert-1) \nonumber\\
& = t_1(2kC^{\text{max}}_0 + 1)\frac{\lVert g-g' \rVert_{\infty}}{C^{\text{min}}_1}(\lvert V \rvert-1) \hspace{3mm} \text{by } g-g^*=t_1(g-g'). \label{eq:2_part}
\end{align}
If we set a constant \(C:=(2kC^{\text{max}}_0 + 1)\frac{1}{C^{\text{min}}_1} (\lvert V \rvert-1) \), then by the inequalities (\ref{eq:1_part}) and (\ref{eq:2_part}), we have 
\begin{align*}
&  \sum_{e \in p^g}\left\lvert L^g(e) - L^{g*}(e) \right\rvert  \\
& \leq \sum_{e \in p^g}\Big\lvert (\widehat{W}^g_E(e))^{-1} -(\widehat{W}^{g^*}_E(e))^{-1} \Big\rvert + \Big\lvert \alpha^g( \rho(\widehat{W}^g_V(a) + \widehat{W}^g_V(b)) - \alpha^{g^*}( \rho(\widehat{W}^{g^*}_V(a) + \widehat{W}^{g^*}_V(b)) \Big\rvert \\
& \leq t_1 \frac{1}{C^{\text{min}}_1} \lVert g - g' \rVert_{\infty} (\lvert V \rvert-1) + t_1(2kC^{\text{max}}_0 + 1)\frac{\lVert g-g' \rVert_{\infty}}{C^{\text{min}}_1}(\lvert V \rvert-1) \\
& = t_1 C  \lVert g-g' \rVert_{\infty}.
\end{align*}
Similarly, we have
\[
\sum_{e \in p^g}\left\lvert L^{g*}(e) - L^{g'}(e) \right\rvert \leq (1-t_1) C  \lVert g-g' \rVert_{\infty}.
\]
Therefore, by the inequality \eqref{eq:total}, we conclude
\[
 \lvert (h_g - h_{g'})(e_*) \rvert \leq t_1 C  \lVert g-g' \rVert_{\infty} + (1-t_1) C  \lVert g-g' \rVert_{\infty} =  C  \lVert g-g' \rVert_{\infty},
\]
which completes the proof.
\section{Applications and experiments}
\label{sec:Applications and experiments}

This section aims to demonstrate the potential applications of adjusting the influence vector in PH analysis of actual time series based on the main theorem demonstrated in Section \ref{sec:Stability theorem for influence vectors}. 
In Sections~\ref{subsec:The first method} and ~\ref{subsec:The second method}, we explore the analysis of stock and music data, respectively.

However, these analyses are not the main focus of this study. 
Instead, we aim to demonstrate that by using our proposed method of adjusting the appropriate influence vector when employing graph representations, it is possible to achieve outcomes comparable to the previous methods for some cases.
Moreover, we highlight the possibility of conducting various other analyses that were not attempted in the previous research. 
The first example aims to improve the time series analysis by adjusting the influence vector and determining its value suitable for analysis. 
For this purpose, we apply our methodology to the stock data analyzed in previous PH studies. 
The second example focuses on analyzing the effect of features on the time series data. 
For this purpose, we apply the proposed method to the music data and observe how the PH results change when the influence assigned to each feature varies.
A more detailed explanation is provided in Section~\ref{subsec:The second method}.

\subsection{Application to stock data}\label{subsec:The first method}

We aim to introduce the process of enhancing the precision in the analysis of time series data by adjusting the influence vector \( g \) based on the stability described in Theorem~\ref{Main_Thm}.
A previous study~\cite{gidea2018topological} employed four time-series, specifically S\&P 500, DJIA, NASDAQ, and Russell 2000, and treated them as the \(4\)-dimensional time series. 
In \cite{gidea2018topological}, the time series was transformed into a point cloud using the sliding window, and the persistence landscape \cite{bubenik2015statistical} was computed. 
Calculating the \( L^p \)-norm of the persistence landscape yields a real number value for anomaly scores. 
Curves were created by advancing the sliding window stepwise and repeating the same process.
In \cite{gidea2018topological}, these curves were used to detect anomalies such as the collapse of the Dotcom bubble on 03/10/2000 and the bankruptcy of Lehman Brothers on 09/15/2008.

For a given date \(t\), define \(P(t)\) as the closing price of a stock. 
Consider a time series $X$ defined as $X(t) = \log P(t+1) - \log P(t)$ for each date $t$.
Because the closing price \(P\) is already aggregated over time \(t\), we further perform an aggregation on the price by partitioning the values of \(X\) into 30 equal intervals, denoted as \(\mathcal{I} = \{ I_{1}, I_{2}, \ldots, I_{30} \} \).
We obtain a aggregated time series \(T: \mathbb{T} \rightarrow \mathcal{I}\) defined by \(T(t) = I_{k_t}\) such that \(X(t) \in I_{k_t}\) for any \(t \in \mathbb{T}\). 
The zeroth feature \((T^0_f)\) is chosen as the day of the week. 
For the first feature \((T^1_f)\), we partition the set \(\{ X(t+1) - X(t) \mid t \in \mathbb{T} \}\) into \(4\) intervals of equal length, denoted as \(\{J_1, J_2, J_3, J_4\}\). 
The time series \(\widehat{T}=(T,T_f)\) is characterized as 
\begin{enumerate}
    \item \(T(t) = I_{k_t}\) such that \(X(t) \in I_{k_t}\) for each date \(t\)
    \item The zeroth feature set $F_0$ used in \(T_f\) is $\{ \emptyset^0, \text{Mon}, \text{Tue}, \text{Wed}, \text{Thu}, \text{Fri} \}$
    \item The first feature set $F_1$ used in \(T_f\) is $\{ \emptyset^1, J_1, J_2, J_3, J_4\}$.
\end{enumerate}

Let \(\widehat{T}^i\) represent the time series obtained from the following indices: S\&P 500, DJIA, NASDAQ, and Russell 2000, for \(i = 1, 2, 3, 4\), respectively.
For each influence vector $g$, we obtain the featured time series \((\widehat{T}^i, g)\).
For a fixed date $t$, define $\widehat{T}^i|_t^w$ as the function obtained by slicing $\widehat{T}^i$ from $t$ to $t + w$, where $w$ is the window size.
Let $\text{dgm}_1(\widehat{T}^i|_t^w)$ denote the one-dimensional PD obtained for the sliced time series $\widehat{T}^i|_t^w$.

We use the persistence landscape~\cite{bubenik2015statistical} as a tool to measure the PD, defined as follows:
Let $\text{dgm}_p$ be a $p$-dimensional PD consisting of points $(b,d)$ where $b$ and $d$ represent the birth and death times of the homological features, respectively. 
First, transform each point $(b_m, d_m)$ in $\text{dgm}_p$ into a tent function,
$\phi_m(x) = \left[ \min\{x - b_m, d_m - x\} \right]^+$, where $[z]^+ = \max\{z, 0\}$ denotes the positive part of $z$. 
The persistence landscape of $\text{dgm}_p$ is defined as the sequence of functions $\lambda_k$, $k \geq 1$, where 
\[
\lambda_k(x) = k^{th} \text{ largest value of } \{\phi_m(x)\}_{m \geq 1}.
\]
If there are fewer than $k$ functions $\phi_m(x)$ at a given point $x$, we set $\lambda_k(x) = 0$. 

We define the anomaly score curve (ASC) of a single time series $P^i$ ranging from $t$ to $t+w$ as
\[
  ASC(P^i)(t) = \frac{1}{M^i}\sum_{k=1}^{\infty} \lVert  \lambda_k( \text{dgm}_1(\widehat{T}^i|_t^w) ) \rVert_{\infty},
\] where \( M^i \) is the normalization factor such that the maximum value of \( ASC(P^i)(t) \) is $1$ and $\lVert \cdot \rVert_{\infty}$ denotes the $L^{\infty}$-norm.
The total anomaly score curve (TASC) of four time-series $P^i$ ranging from $t$ to $t+w$ is defined as
\[
  TASC(t) = \frac{1}{M}\prod_{i=1}^4 ASC(P^i)(t),
\] where \( M \) is the normalization factor such that the maximum value of \( TASC(t) \) is $1$.
The rationale behind taking the product of the four $ASC(P^i)$ is to ensure that the total score approaches zero if any of them is close to zero. 
This is a cautious approach to minimize the error in the anomaly score.

According to the stability theorem for influence vectors proved in Section \ref{sec:Stability theorem for influence vectors}, slight variations in the influence vector $g$ cause subtle changes in the ASC and TASC. 
Experiments were conducted to identify $g$ that enhances the anomaly detection performance. 
The components of $g$ were altered to $0$, $10$, and $20$ to observe significant changes.
Because the zeroth and first features are always present in the stock data for each date $t$, the influence related to $\emptyset^0$ and $\emptyset^1$ was set to $0$.
In Figure~\ref{fig:Anormaly_ASC_TASC}, the date 03/10/2000 is marked as the starting point of the collapse of the Dotcom bubble, and the shaded areas are colored according to the window size from the starting point.
Similarly, 09/15/2008 corresponds to the bankruptcy of Lehman Brothers. 
The values of $\vec{g}_0$ represent the influence values for $(\emptyset^0, \text{Mon}, \text{Tue}, \text{Wed}, \text{Thu}, \text{Fri})$, whereas those of $\vec{g}_1$ correspond to $(\emptyset^1, J_1^i, J_2^i, J_3^i, J_4^i)$.

\begin{figure}[ht]
\centering
\includegraphics[width=\linewidth]{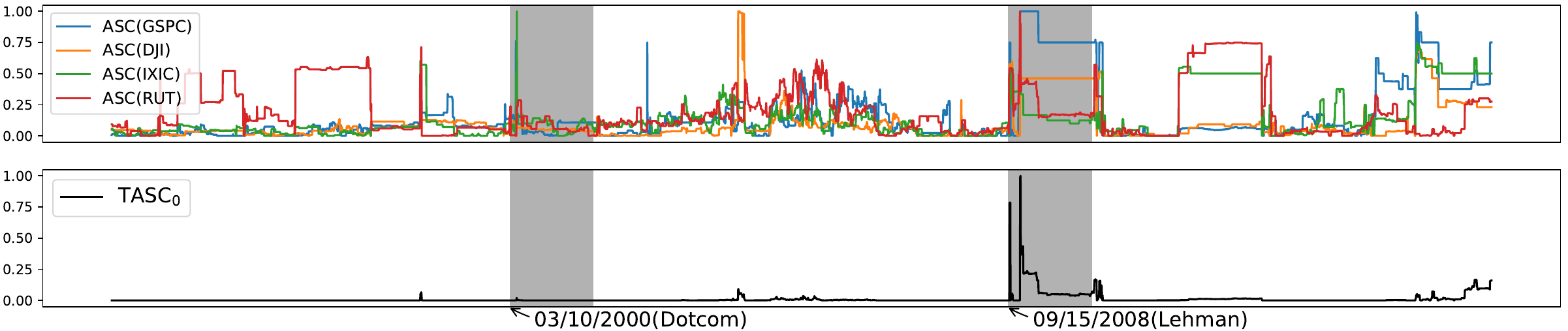}
\caption{The first row presents the ASC for each of the S\&P 500 (GSPC), DJIA (DJI), NASDAQ (IXIC), and Russell 2000 (RUT) indices, with \(\vec{g_0} = \vec{0}\), \(\vec{g_1} = \vec{0}\), and the window size $w=360$. 
The second row shows the TASC$_0$ derived from these four ASCs.}
\label{fig:Anormaly_ASC_TASC}
\end{figure}

As shown in the second row of Figure~\ref{fig:Anormaly_ASC_TASC}, the TASC$_0$ successfully detects the Lehman crisis but fails to detect the Dotcom bubble collapse.

\begin{figure}[ht]
\centering
\includegraphics[width=\linewidth]{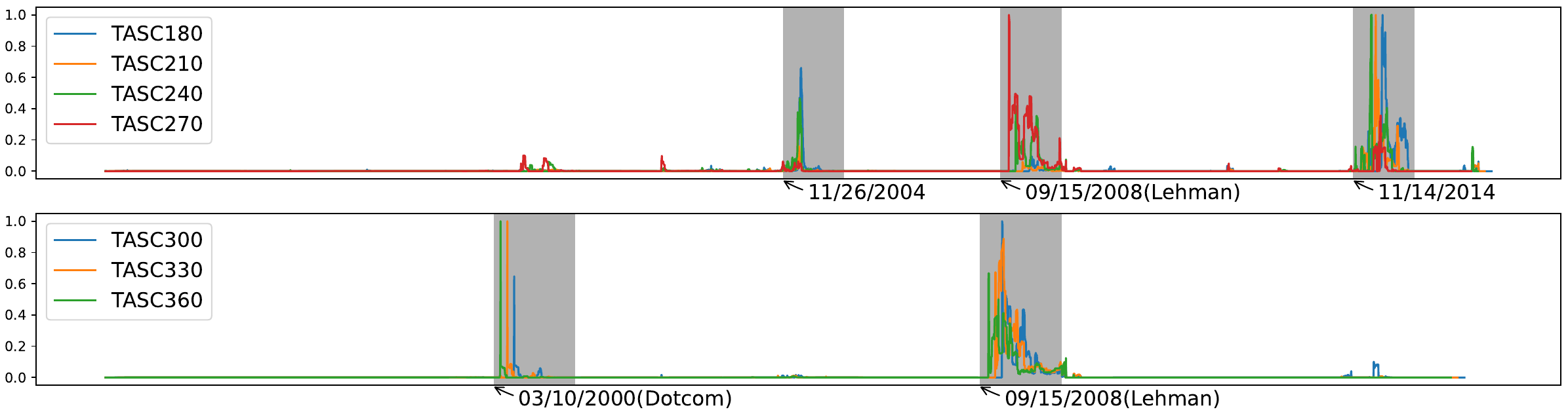}
\caption{TASCs for varying window sizes \(w\) with \(\vec{g_0} = (0,20,0,0,0,0)\) and \(\vec{g_1} = (0,0,10,0,20)\) for each $w \in \{ 180, 210, 240, 270, 300, 330, 360 \}$.}
\label{fig:Anormaly_along_windowsize_good}
\end{figure}

Figure~\ref{fig:Anormaly_along_windowsize_good} shows the TASC results for various window sizes \(w\). 
For window sizes \(w = 180, 210, 240, \text{ and } 270\), the TASCs fail to detect the Dotcom bubble. 
With larger window sizes of \(300, 330, \text{ and } 360\), the TASCs identify the Dotcom bubble area.
For window sizes of \(w = 180, 210, \text{ and } 240\), high anomaly scores are observed in the regions starting at 11/26/2004 and 11/14/2014. 
Increasing the window size to \(270, 300, 330, \text{ and } 360\) results in the disappearance of high anomaly scores in these regions. 
This suggests that the Dotcom and Lehman crises, which have led to sustained a global economic crisis, are predominantly measured when the window size is extended.

\begin{figure}[ht]
\centering
\includegraphics[width=\linewidth]{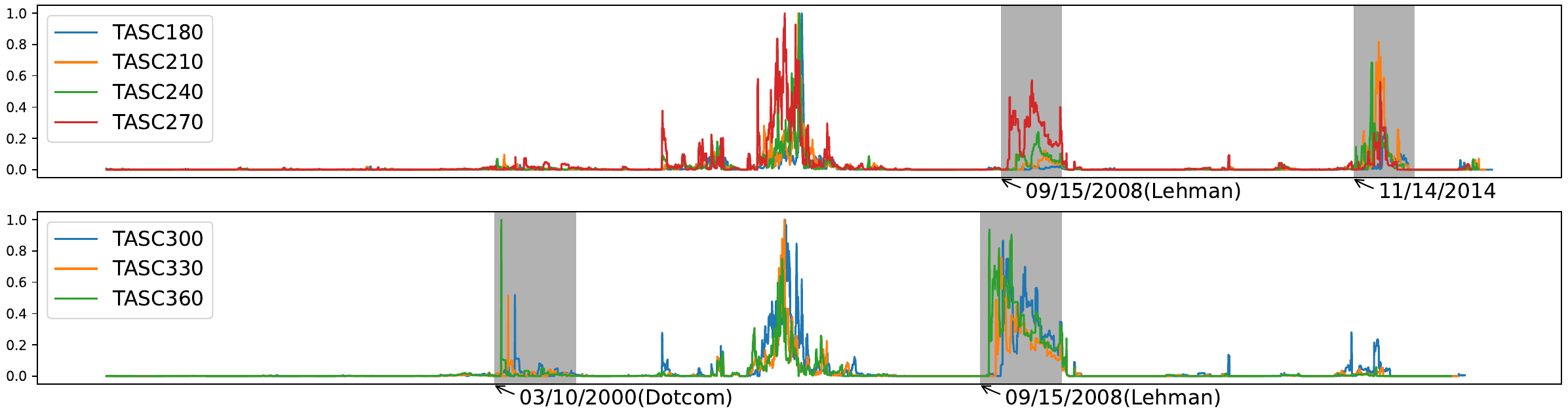}
\caption{TASCs for varying window sizes \(w\) with \(\vec{g_0} = \vec{0} \) and \(\vec{g_1} = (0,0,10,0,20)\) for each $w \in \{ 180, 210, 240, 270, 300, 330, 360 \}$.}
\label{fig:Anormaly_along_windowsize_bad}
\end{figure}

Figure~\ref{fig:Anormaly_along_windowsize_bad} illustrates the outcomes of the same experiment but with \(\vec{g_0}\) set as the zero vector. 
The anomaly scores in the shaded areas resemble those in Figure~\ref{fig:Anormaly_along_windowsize_good}, but noisier anomalies were observed in the mid-region, as shown in Figure~\ref{fig:Anormaly_along_windowsize_bad}.
One approach to analyze the stock volatility is to examine the prices only for a specific day of the week and analyze them weekly.
This approach reduces short-term, repetitive noise and focuses more on long-term trends and patterns. 
In our methodology, by setting \(\vec{g_0} = (0,20,0,0,0,0)\), we emphasize the information for Monday to enable a focused analysis of the effect of this particular day of the week.
Similar results are observed when focusing on Tuesday, Wednesday, Thursday, and Friday.
As demonstrated here, the proposed method generates various domain-specific and interpretable results by adjusting both $ g_0$ and $ g_1$.

\subsection{Application to music data}\label{subsec:The second method}

Music data are time series data containing information beyond pitch and rhythm.
Although music is a challenging form of art to quantify, the authors of \cite{tran2023topological} analyzed the hidden topological information in traditional Korean music.
We analyze the Celebrated Chop Waltz as a simple example, where each note is played simultaneously with both the right and left hands.
Define the pair of pitches played simultaneously using the left and right hands as a \textit{pitch pair}.
We define the time series $T$ as $T(t_m) = (P(t_m), D(t_m))$, where $P(t_m)$ is a pitch pair and $D(t_m)$ is the playing length of each note played at the $m$th timestamp $t_m$.
For instance, in the left part of Figure~\ref{fig:musical notation}, the pitch pair $(F, G)$ is repeated six times with an eighth note duration.

\begin{figure}[ht]
\centering
\includegraphics[width=0.8\linewidth]{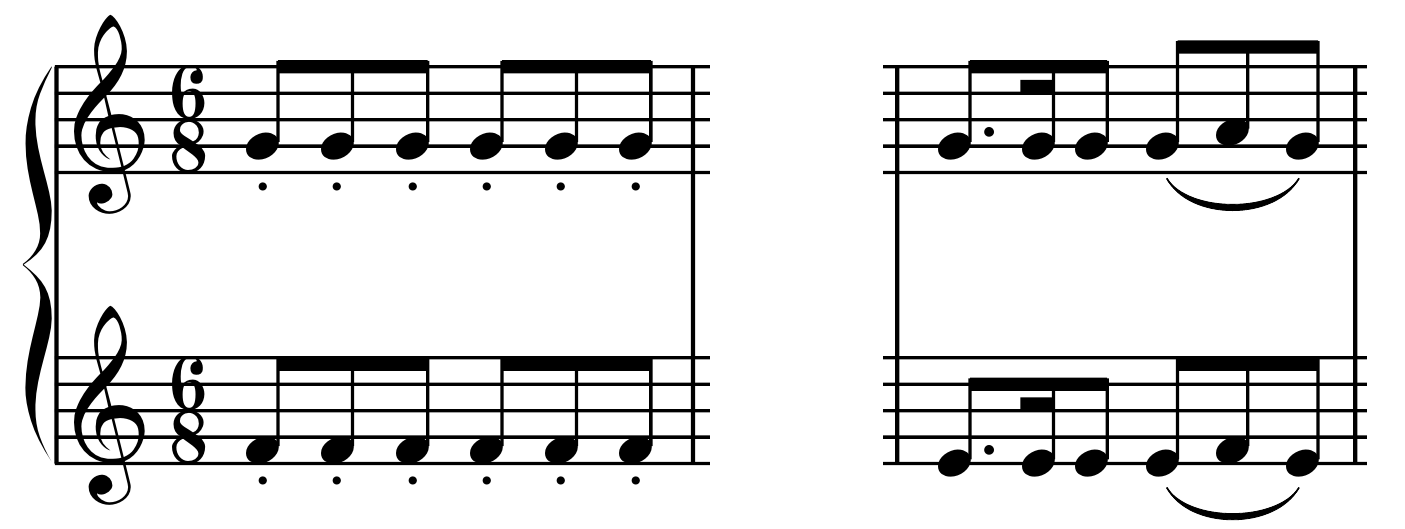}
\caption{Musical notations used in the Celebrated Chop Waltz. (Left) When a staccato symbol is attached, the note is played briefly as if it were popping. 
    (Right) If multiple notes are tied together with a slur symbol, these notes are played smoothly and connectedly.}
\label{fig:musical notation}
\end{figure}

As shown in Figure~\ref{fig:musical notation}, the Celebrated Chop Waltz contains staccato and slur musical notations. 
We aim to understand how the changes in the influence of these musical notations affect the results of PH calculations.
Staccato is a technique in which notes are played very briefly.
Slur indicates that two distinct notes should be played as if they are connected smoothly. 
We define staccato related to a single note as the zeroth feature and slur related to two or more notes as the first feature.

The time series $\widehat{T}=(T,T_f)$ is characterized as
\begin{enumerate}
    \item $T(t_m) = (P(t_m), D(t_m))$ at the timestamp $t_m$.
    \item The zeroth feature set $F_0$ used in \(T_f\) is $\{ \emptyset^0, \text{ staccato} \}$.
    \item The first feature set $F_1$ used in \(T_f\) is $\{ \emptyset^1, \text{ slur} \}$.
\end{enumerate}

\begin{figure}[ht]
    \centering
    \subfloat[\label{fig:The longest persistence of barcode auto}]{
        \includegraphics[width=0.45\linewidth]{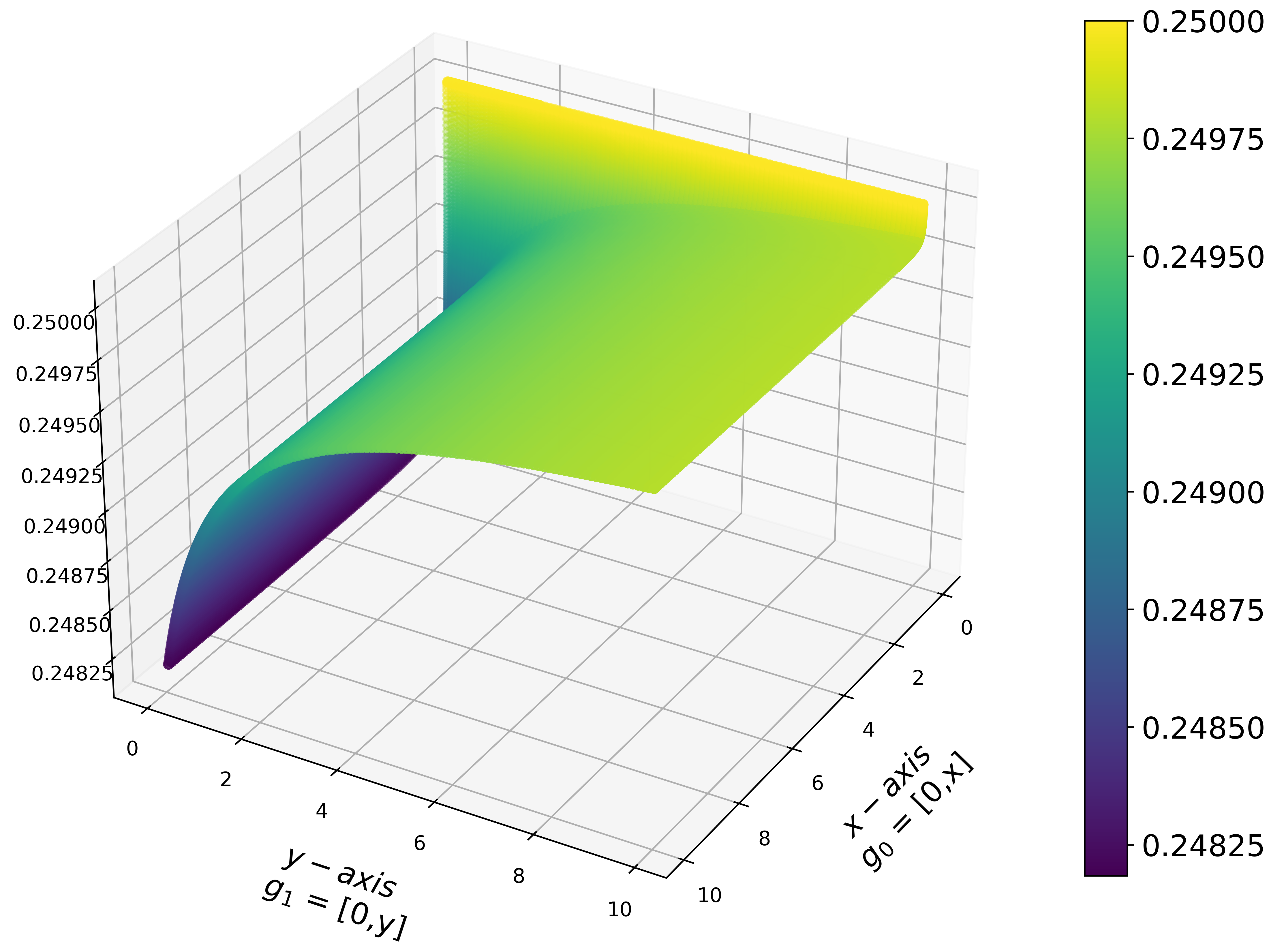}
    }
    \hspace{2mm}
    \subfloat[\label{fig:The shortest persistence of barcode auto}]{
        \includegraphics[width=0.45\linewidth]{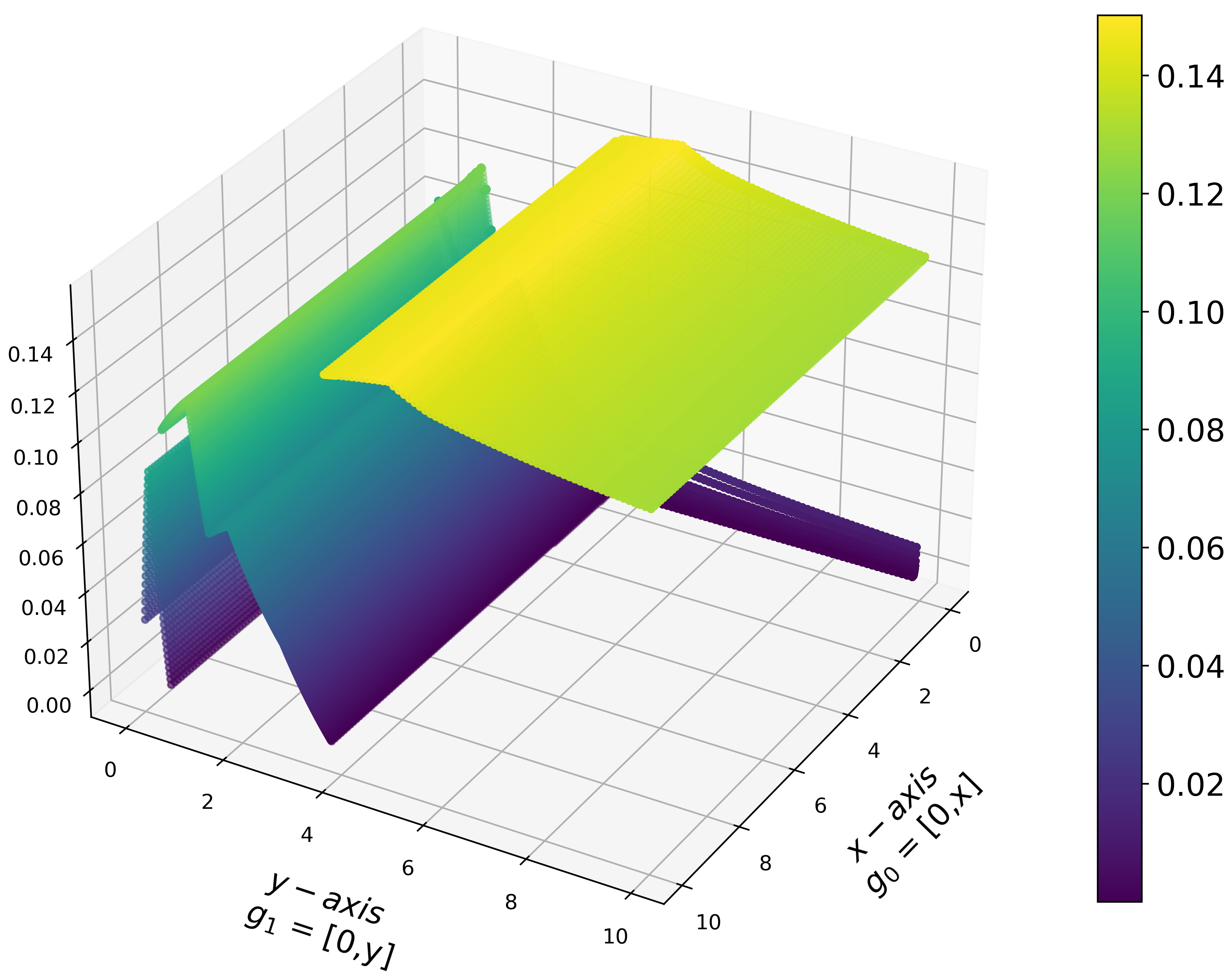}
    }\\
    \subfloat[\label{fig:The number of barcode auto}]{
        \includegraphics[width=0.45\linewidth]{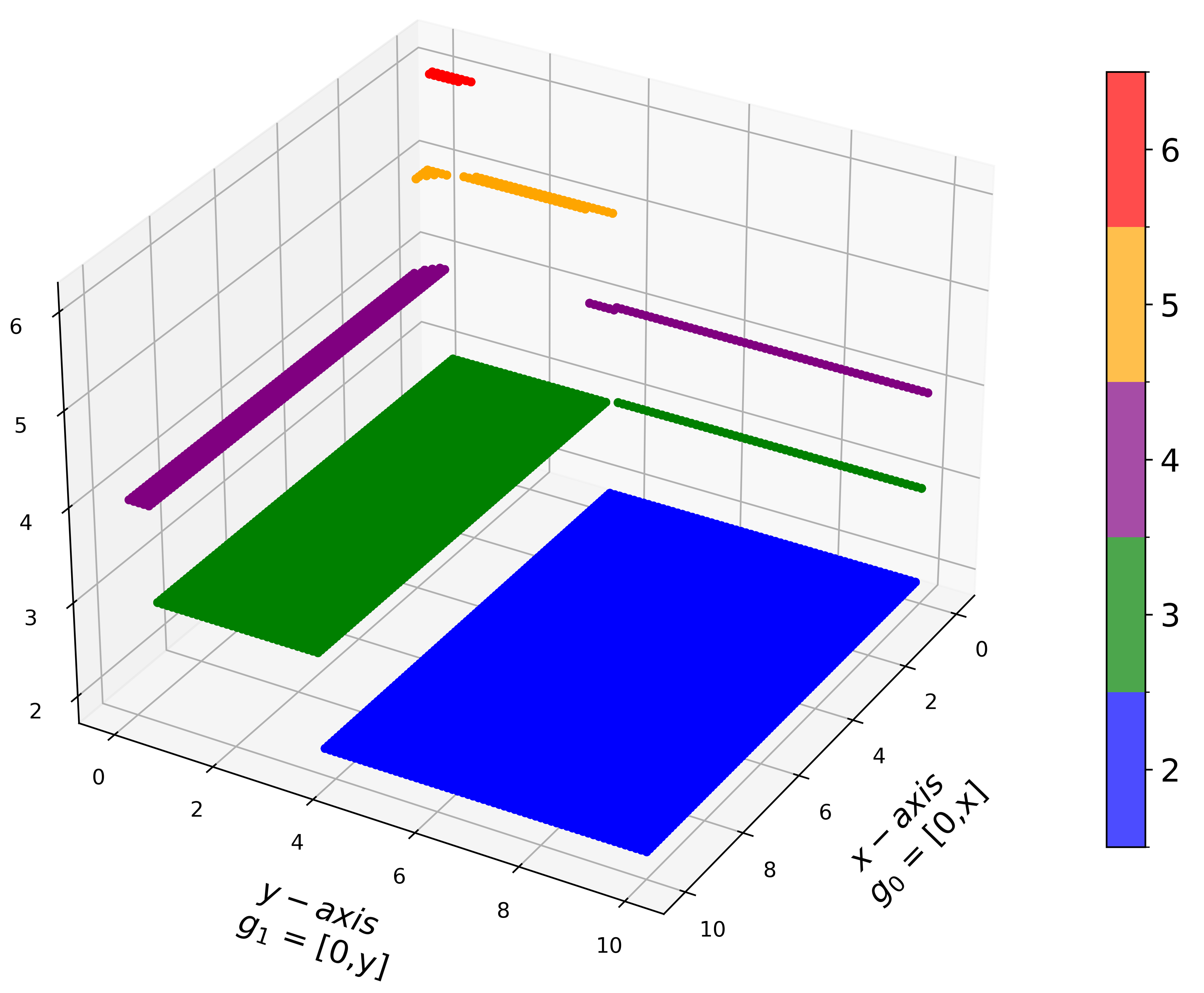}
    }
    \hspace{2mm}
    \subfloat[\label{fig:infnorm_dgm1 auto}]{
        \includegraphics[width=0.45\linewidth]{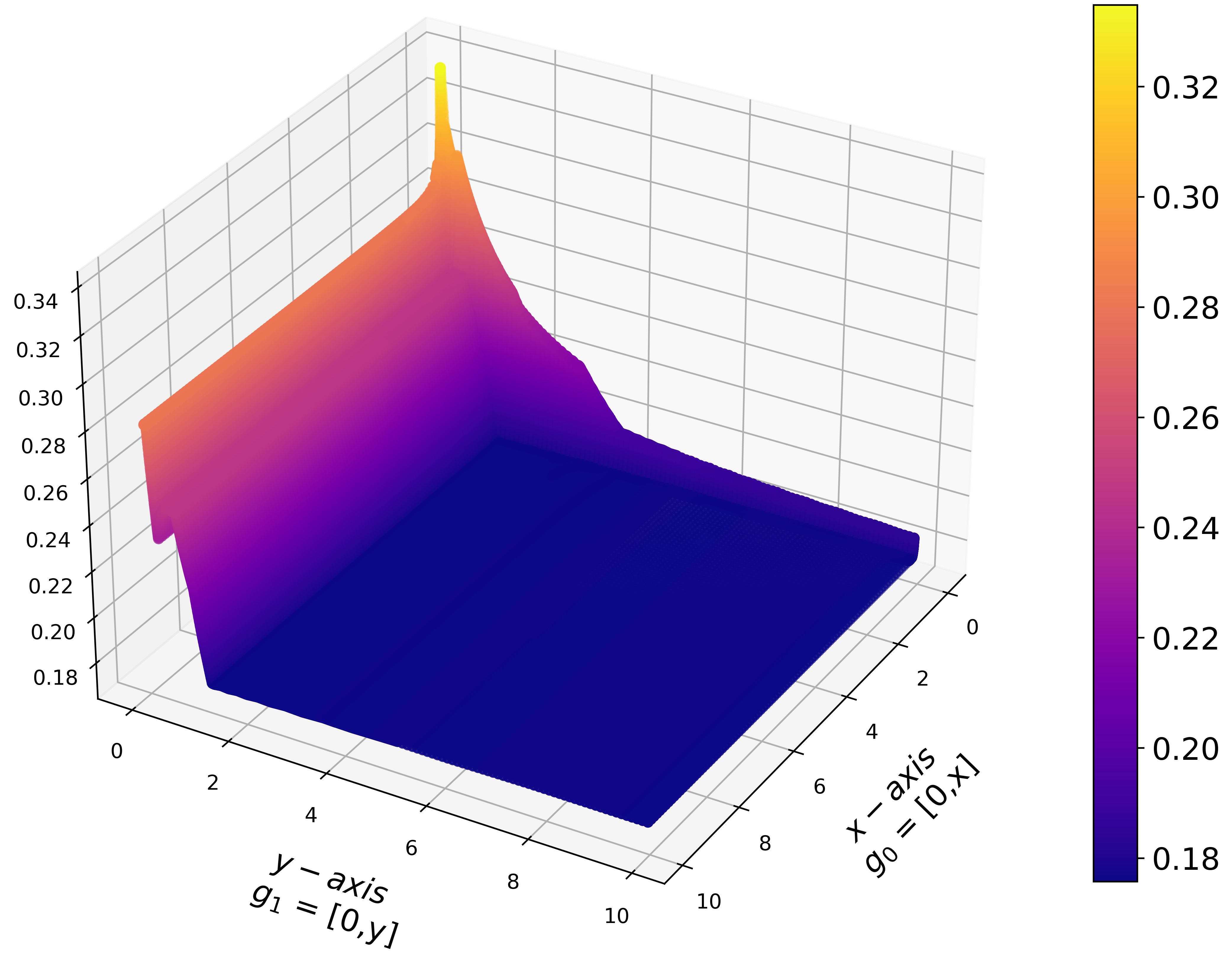}
    }
    \caption{Topological information about the one-dimensional persistence diagrams (PDs) while the influence vectors $g$ of the staccato and slur are varied.
    \textup{(a)} The longest persistence among the elements of the one-dimensional PD as $g$ is varied.
    \textup{(b)} The shortest persistence among the elements of the one-dimensional PD as $g$ is varied.
    \textup{(c)} The total number of points in the one-dimensional PD as $g$ is varied.
    \textup{(d)} The $L^1$-norm of the landscape of the one-dimensional PD as $g$ is varied.}
    \label{fig:PH_music_experiment}
\end{figure}

We conducted experiments to investigate the effects of changes in influence values on the resulting PDs. 
Figure~\ref{fig:PH_music_experiment} shows the topological information about the one-dimensional PD, $\text{dgm}_1$, as influence values for the staccato and slur vary.
In Figure~\ref{fig:PH_music_experiment}, the $x$-axis represents the cases where $\vec{g_0}$ is $(0,x)$, which corresponds to the influence of the staccato being $x$. 
The $y$-axis signifies when $\vec{g_1}$ is $(0,y)$, indicating the value of the slur as $y$. 
Figures~\ref{fig:The longest persistence of barcode auto} and~\ref{fig:The shortest persistence of barcode auto} plot the longest and shortest persistence among all points in $\text{dgm}_1$, respectively.
Here, the persistence is defined as \( d - b \) for each point $(b, d) \in \text{dgm}_1(\vec{g_0},\vec{g_1})$. 
The function $z = z(x,y)$ in Figure~\ref{fig:The longest persistence of barcode auto} is 
\[
z(x,y) = \max \{ d-b | \; (b, d) \in \text{dgm}_1(\vec{g_0},\vec{g_1}) \} \text{ with } \vec{g_0}=(0,x) \text{ and } \vec{g_1}=(0,y)
\] and that in Figure~\ref{fig:The shortest persistence of barcode auto} is
\[
z(x,y) = \min \{ d-b | \; (b, d) \in \text{dgm}_1(\vec{g_0},\vec{g_1}) \} \text{ with } \vec{g_0}=(0,x) \text{ and } \vec{g_1}=(0,y).
\]
Figure~\ref{fig:The number of barcode auto} shows the total number of points in $\text{dgm}_1(\vec{g_0},\vec{g_1})$. 
The function $z = z(x,y) $ in Figure~\ref{fig:The number of barcode auto} is defined as 
\[
z(x,y) = \lvert \text{dgm}_1(\vec{g_0},\vec{g_1}) \rvert \text{ with } \vec{g_0} = (0,x) \text{ and } \vec{g_1}=(0,y),
\]
where $\lvert \; \cdot \; \rvert $ represents the number of elements in the multiset.
Figure~\ref{fig:infnorm_dgm1 auto} illustrates the sum of the $L^{\infty}$-norm of the $k$-th persistence landscape $\lambda_k$ of $\text{dgm}_1(\vec{g_0},\vec{g_1})$, which can be considered as information related to the overall total persistence.
The function $z = z(x,y)$ in Figure~\ref{fig:infnorm_dgm1 auto} is 
\[
z(x,y) =  \sum_{k=1}^{\infty} \lVert\lambda_k(\text{dgm}_1(\vec{g_0},\vec{g_1})) \rVert_{\infty} \text{ with } \vec{g_0}=(0,x) \text{ and } \vec{g_1}=(0,y).
\]
The continuity of the surface observed in Figure~\ref{fig:infnorm_dgm1 auto} is a consequence of Theorem~\ref{Main_Thm}.

From Figure~\ref{fig:The number of barcode auto}, we observe that when $(x,y)=(0,0)$, the number of points in $\text{dgm}_1$ is the highest.
However, as $x$ and $y$ increase, the number of points gradually decreases. 
If \(x\) and \(y\) are sufficiently large, only the two points with longer persistence remain.
The disappearing points with smaller persistence cause the discontinuities shown in Figure~\ref{fig:The shortest persistence of barcode auto}.
In Figures~\ref{fig:The longest persistence of barcode auto} and~\ref{fig:The shortest persistence of barcode auto}, the variation in the shortest persistence is greater than that in the longest persistence. 
Therefore, the changes in Figure~\ref{fig:infnorm_dgm1 auto} are primarily due to the variation of the points with shorter persistence.
This suggests that points with shorter persistence can provide important information about the overall topological structure of the music.
The authors of \cite{tran2023topological} also considered all points representing the music's hidden structure as important, regardless of the persistence. 
From this perspective, when the influences of staccato and slur are zero, the highest number of hidden topological structures of the music are observed.
However, when the influences of staccato and slur are sufficiently large, only two structures remain.
In summary, starting with six points when $(x,y)=(0,0)$, changes in the $x$ and $y$ values cause points with relatively shorter persistence to vary dramatically, with up to four disappearances. 
This illustrates how the overall PH outcome changes with variations in the influence of staccato and slur.

\begin{figure}[ht]
    \centering
    \includegraphics[width=0.5\linewidth]{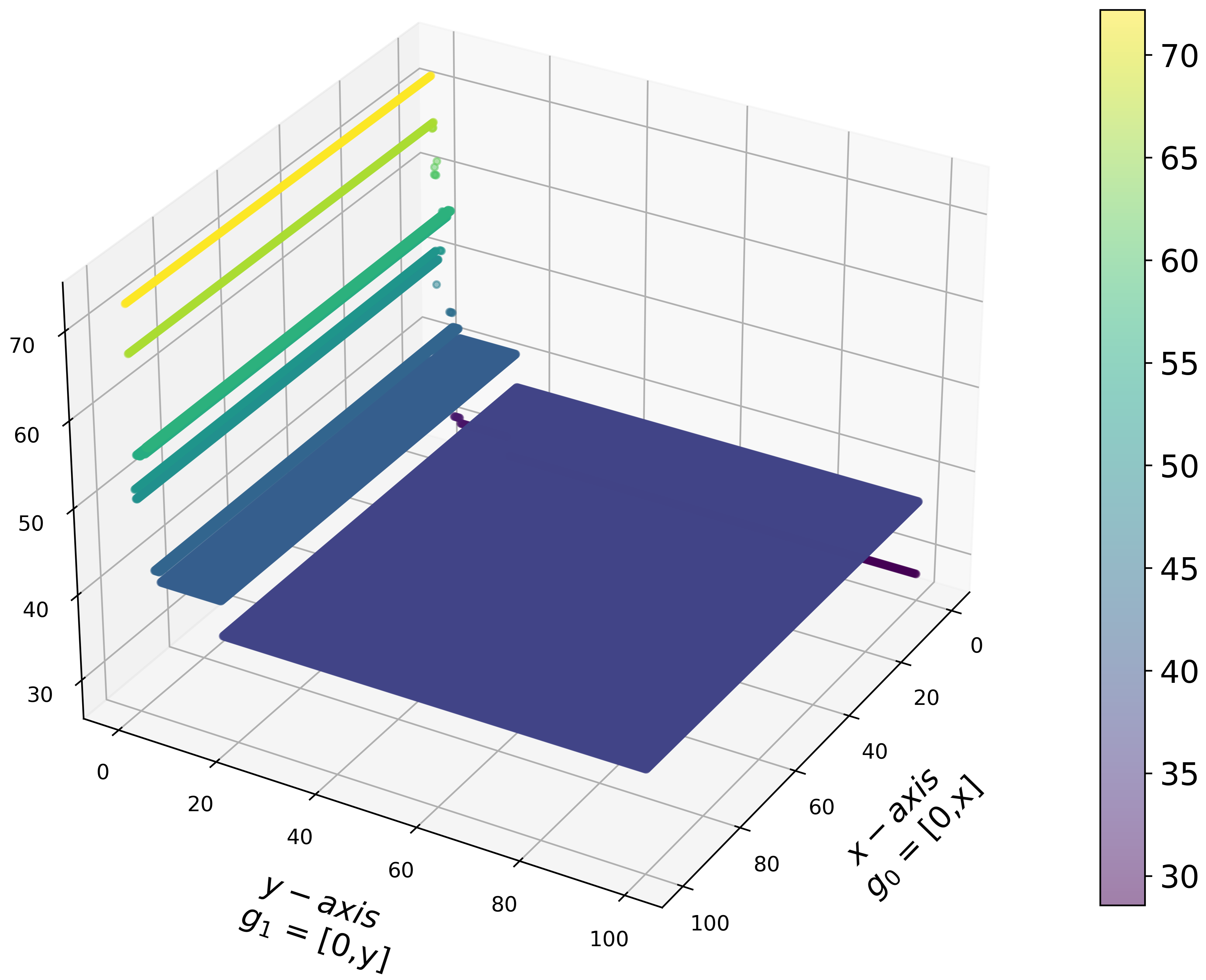}
    \caption{The z-axis represents the overlapping percentage for each influence vector.}
    \label{fig:The overlapping_percentage}
\end{figure}

In \cite{tran2023topological}, the authors analyzed the hidden topological structures in music using the cycle representatives of each point in $\text{dgm}_1$. 
They introduced the \textit{overlapping percentage} to quantify the topological repetitiveness in music. 
Assume that $\text{dgm}_1$ contains $m$ points $[b_1,d_1], \ldots, [b_m, d_m]$. 
For each point $[b_i,d_i]$, we fix a corresponding one-dimensional cycle representative $\sigma_i \subseteq V$. 
Although $\sigma_i$ is generally not unique, the authors of \cite{tran2023topological} used the JavaPlex algorithm \cite{adams2014javaplex} to select $\sigma_i$. 
The chosen cycle representatives $\sigma_i$ correspond closely with the $2$-simplex, which directly results in the disappearance of the $1$-cycle.
For a music time series $T: \mathbb{T} \rightarrow \mathcal{X}$, define 
\[
N_c = \big\lvert \{ t \in \mathbb{T} \mid T(t) \in \bigcup_{i=1}^m \sigma_i \} \big\rvert \;\text{ and }\;  
N_s = \big\lvert \{ t \in \mathbb{T} \mid T(t) \in \bigcup_{1 \le i < j \le m} (\sigma_i \cap \sigma_j) \} \big\rvert.
\]
Then, the \textit{overlapping percentage} is defined as 
\(
(N_s/N_c) \times 100 (\%)
\).
This value represents the percentage of timestamps in a time series that belongs simultaneously to multiple cycle representatives. 
Figure~\ref{fig:The overlapping_percentage} shows the overlapping percentage obtained by varying the influence vector $g$.
In Figure~\ref{fig:The overlapping_percentage}, the overlapping percentage increases for the staccato influence $x$ when $x>0$ than $x=0$. 
This suggests that the presence of staccato adds greater topological repetitiveness than its absence; however, increasing the value does not necessarily enhance it further. 
Moreover, as the influence of the slur $y$ increases, the overlapping percentage decreases. Slurs appear less frequently in scores than staccato. 
Emphasizing the slurs gradually decreases the overlapping percentage, reducing the topological repetitiveness.

In this section, we described the effect of adjusting the influence vector $g$ on the overall PH outcome. 
This analysis allows us to examine the impact of features on the time series.
As shown in this section, the proposed method yields a more flexible analysis by varying $g$ from the music data.
\section{Conclusion}\label{sec:Conclusion}
In this paper, we introduced a novel concept of the featured time series data and proposed a method based on persistent homology with varying influence vectors. 
With the featured time series and influence vectors, 
the proposed method allows customizing PH calculations to reflect domain-specific knowledge effectively. 
Our approach enhances the adaptability of domain-specific PH analysis across different types of time series data and ensures the stability of the PH calculations.
The application of our methodology to real-world datasets, such as stock data for anomaly detection and musical data for feature impact analysis, has shown promising results. 
These applications demonstrate the practical feasibility and effectiveness of our proposed method. 
Future research will focus on extending this framework to other applications and exploring the optimization method of the influence vectors that can further refine the accuracy and applicability of PH under various settings. 

\bibliographystyle{unsrt}  
\bibliography{references} 

\end{document}